\documentclass[10pt]{article}
\usepackage{etex}
\usepackage[margin={1.2in}]{geometry}
\usepackage[titletoc]{appendix}

\usepackage{enumerate}

\usepackage{tikz,eso-pic}
\usepackage{tikz-3dplot}
\usetikzlibrary{cd,matrix, calc, arrows,
decorations.pathreplacing, decorations.markings, decorations.pathmorphing,
backgrounds,fit,positioning,shapes.symbols,chains,shadings,fadings,intersections}
\usetikzlibrary{trees}

\tikzset{->-/.style={decoration={  markings,  mark=at position #1 with
    {\arrow{>}}},postaction={decorate}}}
\tikzset{-<-/.style={decoration={  markings,  mark=at position #1 with
    {\arrow{<}}},postaction={decorate}}}

\usepackage{tikz}
\usetikzlibrary{matrix}

\usepackage{eucal}
\usepackage{mathrsfs}
\usepackage{stmaryrd}

\usepackage{hyperref} 

\usepackage{rotating}

\usepackage{longtable}

\usepackage[sc]{mathpazo}
\usepackage{courier}

\usepackage[utf8]{inputenc}
\usepackage[T1]{fontenc}

\usepackage[english]{babel}
\usepackage{blindtext}

\usepackage{amsmath}

\usepackage{microtype}

\usepackage{amsmath}
\usepackage{amssymb}
\usepackage{amsthm}
\usepackage{color}
\usepackage{latexsym,extarrows}

\usepackage[all]{xy}
\usepackage{graphicx}
\usepackage{subcaption}

\usepackage[stable,multiple]{footmisc}

\RequirePackage[font=small,format=plain,labelfont=bf,textfont=it]{caption}
\makeatletter
\newsavebox{\@brx}
\newcommand{\llangle}[1][]{\savebox{\@brx}{\(\m@th{#1\langle}\)}%
  \mathopen{\copy\@brx\kern-0.5\wd\@brx\usebox{\@brx}}}
\newcommand{\rrangle}[1][]{\savebox{\@brx}{\(\m@th{#1\rangle}\)}%
  \mathclose{\copy\@brx\kern-0.5\wd\@brx\usebox{\@brx}}}
\makeatother

\def\e#1\e{\begin{equation}#1\end{equation}}
\def\ea#1\ea{\begin{align}#1\end{align}}

\theoremstyle{plain}
\newtheorem{thm}{Theorem}[section]
\newtheorem{lem}[thm]{Lemma}
\newtheorem{prop}[thm]{Proposition}
\newtheorem{cor}[thm]{Corollary}
\theoremstyle{definition}
\newtheorem{dfn}[thm]{Definition}
\newtheorem{ex}[thm]{Example}
\newtheorem{rem}[thm]{Remark}
\newtheorem{conv}[thm]{Convention}

\newcommand{\A}{\mathcal{A}}

\newcommand{\op}{\operatorname}
\newcommand{\C}{\mathbb{C}}

\newcommand{\Z}{\mathbb{Z}}
\renewcommand{\H}{\mathbf{H}}

\newcommand{\im}{\op{im}}

\newcommand{\bracket}[1]{\left(#1\right)}

\newcommand{\pa}{\partial}

\newcommand{\dbar}{\bar\pa}

\DeclareMathOperator{\Res}{Res}

\def\Xint#1{\mathchoice
{\XXint\displaystyle\textstyle{#1}}%
{\XXint\textstyle\scriptstyle{#1}}%
{\XXint\scriptstyle\scriptscriptstyle{#1}}%
{\XXint\scriptscriptstyle\scriptscriptstyle{#1}}%
\!\int}
\def\XXint#1#2#3{{\setbox0=\hbox{$#1{#2#3}{\int}$}
\vcenter{\hbox{$#2#3$}}\kern-.5\wd0}}

\def\dashint{\Xint-}

\numberwithin{equation}{section}

\makeatletter
\newcommand{\subjclass}[2][2010]{%
  \let\@oldtitle\@title%
  \gdef\@title{\@oldtitle\footnotetext{#1 \emph{Mathematics Subject Classification.} #2}}%
}
\newcommand{\keywords}[1]{%
  \let\@@oldtitle\@title%
  \gdef\@title{\@@oldtitle\footnotetext{\emph{Key words and phrases.} #1.}}%
}
\makeatother

\makeatletter
\let\orig@afterheading\@afterheading
\def\@afterheading{%
   \@afterindenttrue
  \orig@afterheading}
\makeatother

\begin{document}
\title{\bf Regularized Integrals on Elliptic Curves and  Holomorphic Anomaly Equations}
\author{Si Li and Jie Zhou}
\date{}
\maketitle

\begin{abstract}
 We derive residue formulas for the regularized integrals (introduced in \cite{Li:2020regularized}) on configuration spaces of elliptic curves.
Based on these formulas, we prove that the
regularized integrals satisfy holomorphic anomaly equations, providing a mathematical formulation of the so-called contact term singularities.
We also discuss residue formulas for the ordered $A$-cycle integrals and establish their relations with those for
 the regularized integrals.
\end{abstract}

\setcounter{tocdepth}{2} \tableofcontents

\section{Introduction}



We start by briefly reviewing the notion of regularized integrals introduced in our earlier work  \cite{Li:2020regularized}.
Let $\Sigma$ be a Riemann surface  and $D$ be a reduced effective divisor.
Let $ \A^2(\Sigma),  \A^2(\Sigma,\star D)$ be the spaces of $2$-forms on $\Sigma$ that are smooth everywhere, smooth on $\Sigma- D$ with possibly holomorphic poles along $D$, respectively.
The regularized integral $\dashint_{\Sigma}$ introduced in our earlier work \cite{Li:2020regularized}
extends the usual integral $\int_\Sigma$ for smooth forms
\[
\xymatrix{
\A^2(\Sigma)\ar@{^{(}->}[rr]  \ar[dr]_{\int_\Sigma} && \A^2(\Sigma,\star D)\ar[dl]^{\dashint_\Sigma}\\
&\C&
}
\]
This can be generalized to integrals on the configuration spaces $\mathrm{Conf}_{n}(\Sigma)=\Sigma^{n}-\Delta_{n}$, where $\Delta_{n}$
is the big diagonal, and one has
\[
\dashint_{\Sigma}: \A^{2n}(\Sigma^n,\star \Delta_{n})
\rightarrow \A^{2n-2}(\Sigma^{n-1}, \star \Delta_{n-1})\,.
\]
The iterated regularized integrals then give rise to
\[
\dashint_{\Sigma^n}: \A^{2n}(\Sigma^n, \star \Delta_{n})\rightarrow \C\,.
\]
This notion of regularized integrals satisfies  Stokes theorem,
Fubini-type theorem, etc.,
and works for family versions, among many other nice properties. In the studies of 2d chiral quantum field theories,  it offers a geometric framework for constructing correlation functions of non-local operators. For example, the composition of chiral conformal blocks with regularized integrals leads \cite{Gui-Li2021} to an elliptic chiral analogue of trace map on chiral algebras. \\


Regularized integrals on the elliptic curve $\Sigma=E=\mathbb{C}/(\mathbb{Z}\oplus \mathbb{Z}\tau)$
enjoy particularly interesting properties.
This is mainly due to the fact that
    functions and differential forms on the configuration space $\mathrm{Conf}_{n}(E)$ of elliptic curves
    admit very concrete descriptions via their lifts to the universal cover that are equivariant under the action of the Jacobi group
    $ (\mathbb{Z}\oplus \mathbb{Z}\tau)\rtimes\mathrm{SL}_{2}(\mathbb{Z}) $.

The above notion of regularized integrals is closely related to the so-called ordered $A$-cycle integrals previously studied in the literature
\cite{Dijkgraaf:1995, Roth:2009mirror, Li:2011mi, Boehm:2014tropical, Goujard:2016counting, Oberdieck:2018}.
Unlike the former, the latter is  automatically convergent. To be more precise, let $\Phi$ be a meromorphic function on $E_{[n]}:=E_{n}\times E_{n-1}\times\cdots\times E_{1}$, where $[n]:=(1,2,\cdots, n)$,
with possibly only  poles along the big diagonal.
Let
$A_{i_n}\times A_{i_{n-1}}\times \cdots \times A_{i_{1}}$ be the lift of the corresponding cycle class in $H_{n}(E_{[n]},\mathbb{Z})$ to the cycle on
$\mathbb{C}^{n}$ that is the product of $n$ disjoint line segments connecting $\varepsilon_{i_k}\tau$ to $\varepsilon_{i_k}\tau+1$, with
$0<\varepsilon_{i_n}<\varepsilon_{i_{n-1}}\cdots < \varepsilon_{i_1}<1$.
The ordered $A$-cycle integral is then defined as
\[\int_{A_{\sigma([n])}}\Phi\, dz_{1}\wedge dz_{2}\cdots \wedge dz_{n}=\int_{A_{\sigma(n)}}\cdots\int_{A_{\sigma(2)}}\int_{A_{\sigma(1)}}\Phi\,dz_1dz_2\cdots dz_n
\,,\quad\sigma\in \mathfrak{S}_{n}\,.\]

\begin{conv}\label{conv3}  Fixing a linear coordinate $z$ on the universal cover $\mathbb{C}$ of $E$, we obtain the holomorphic 1-form $dz$
and the volume form $\mathrm{vol}:={\sqrt{-1}\over 2 \mathrm{Im}\,\tau}dz\wedge d\bar{z}$ on $E$. For simplicity, we will sometimes omit them and use short-hand notations such as
$\dashint_{E} f:=\dashint_{E} f\,\mathrm{vol}, \int_{A}f:=\int_{A}f dz$.
Similarly, we will write
$\Res_{z=p}(f)$
for the residue  $\Res_{z=p}(f dz)$  of the corresponding differential $fdz$ at $z=p$, and write $\Res(f)$ for the summation of all local residues.
\end{conv}

When applied to such a function $\Phi$, the regularized integral
\[\dashint_{E_{[n]}}:=\dashint_{E_{n}}\cdots\dashint_{E_{2}}\dashint_{E_{1}}\]
gives rise \cite{Li:2020regularized}   to an almost-holomorphic modular form \cite{Kaneko:1995}. Furthermore, it
 is exactly the modular completion (in the sense of Kaneko-Zagier \cite{Kaneko:1995}) of the average over orderings of the ordered $A$-cycle integrals
\[\int_{A_{\sigma([n])}}:=\int_{A_{\sigma(n)}}\cdots\int_{A_{\sigma(2)}}\int_{A_{\sigma(1)}}\,,\quad\sigma\in \mathfrak{S}_{n}\,.\]
That is, one has \cite[Theorem 3.4]{Li:2020regularized}
\begin{equation}\label{eqnaveraging}
\lim_{\mathbold{Y}=0}\dashint_{E_{[n]}}\Phi
=
{1\over n!}
 \sum_{\sigma\in \mathfrak{S}_{n}}\int_{A_{\sigma([n])}}\Phi\,,\quad \mathbold{Y}=-{\pi\over \mathrm{im}\,\tau}\,,
\end{equation}
 where $\lim_{\mathbold{Y}=0}$ on the left hand side stands for the holomorphic limit map \cite{Kaneko:1995} on almost-holomorphic modular forms.
This relation leads to an alternative proof  \cite[Theorem 3.9]{Li:2020regularized} of the mixed-weight phenomenon of the
quasi-modularity of
ordered $A$-cycle integrals \cite{Goujard:2016counting, Oberdieck:2018},
and explains both conceptually and combinatorially why averaging the ordered $A$-cycle integrals gives
quasi-modular forms of pure weight  \cite{Oberdieck:2018}.

One interesting consequence is that, based on  standard facts from theory \cite{Kaneko:1995} of quasi-modular and almost-holomorphic modular forms,
any regularization scheme whose leading part produces the averaged ordered $A$-cycle integrals
and which generalizes to family versions must be equivalent to our notion of regularized integrals.\\

The present work studies further properties of the regularized integrals on elliptic curves and their applications.
The main result  is the following holomorphic anomaly equation.
\begin{thm}[=Theorem \ref{thmHAE}]
Let $\Psi$ be an almost-elliptic function, then one has
\begin{equation}\label{eqnHAE}
\partial_{\mathbold{Y}}\dashint_{E_{[n]}}\Psi
=\dashint_{E_{[n]}} \partial_{\mathbold{Y}}\Psi -\sum_{a,b:\,a< b}
\dashint_{E_{[n]-\{a\}}} \Res_{z_{a}=z_{b}}( (z_{a}-z_{b})\Psi)\,.
\end{equation}
\end{thm}

Combining this with \eqref{eqnaveraging}, we obtain the following closed-form expression for the
regularized integration on elliptic functions.

\begin{thm}[=Theorem \ref{thmHAEBCOV}]
As operators on elliptic functions, the following holds
\begin{equation}\label{eqnHAEintegratedintro}
\dashint_{E_{[n]}}=-{1\over 2}\mathbold{Y} \dashint S
-({1\over 2})^{2}{1\over 2!}\mathbold{Y}^{2} \dashint SS
-({1\over 2})^3{1\over 3! }\mathbold{Y}^{3} \dashint SSS+\cdots
+{1\over n!}\sum_{\sigma\in\mathfrak{S}_{n}}\int_{A_{\sigma([n])}}\,,
\end{equation}
where
\[
S(-)=\sum_{a}\sum_{ r:\, r\neq a}\Res_{z_{a}=z_{r}}(z_{ar}-)\,,\quad z_{ar}=z_{a}-z_{r}\,.\]
\end{thm}

The expression \eqref{eqnHAEintegratedintro}
resembles the form
of holomorphic anomaly  in the physics literature
 \cite{Bershadsky:1993cx, Dijkgraaf:1997chiral} and offers a mathematical formulation of the so-called contact term singularities
 that are systematically studied  by Djikgraaf \cite{Dijkgraaf:1997chiral} from the physics perspective
(see also  \cite{rudd1994string, Douglas:1995conformal}).

\subsection*{Structure of the paper}

In Section 2 we derive residue formulas for regularized integrals.
In Section 3
we prove that regularized integrals satisfy holomorphic anomaly equations and discuss their various reformulations and interpretations.
In Section 4 we derive residue formulas for ordered $A$-cycle integrals, and discuss some of their consequences by comparing them to
those for the regularized integrals.
In Appendix \ref{appendixcombinatorial} we discuss some combinatorial properties of residue formulas and their applications.

\subsection*{Acknowledgment}

The authors would like to thank Robbert Dijkgraaf, Zhengping Gui,  Xinxing Tang and Zijun Zhou for helpful communications and discussions.
Both authors are  supported by
the  national key research and development
program of China (NO. 2020YFA0713000).
J.~Z. is also partially supported by
the Young overseas high-level talents introduction plan of China.

\section{Residue formulas for regularized integrals on elliptic curves}
\label{secresidueformulasforregularizedintegrals}

\subsection{Regularized integrals}

We briefly recall the notion of regularized integrals on the configuration spaces of Riemann surfaces as introduced in \cite{Li:2020regularized}.

 Let $\Sigma$ be a compact Riemann surface, possibly with  boundary $\pa\Sigma$. Let $D\subseteq \Sigma$ be a finite subset of points which do not meet $\pa \Sigma$.  Let $\omega$ be a $2$-form on $\Sigma$ which is smooth away from $D\subseteq\Sigma$ but may admit holomorphic poles of arbitrary order along $D$. To be more precise, let $z$ be a local holomorphic coordinate around a point $p\in D$ located at $z=0$. Then $\omega$ is locally expressed as $\omega= {\alpha\over z^n}$, where $\alpha$ is a smooth 2-form and $n$ is some integer.

 Let us decompose $\omega$ as a sum (such a decomposition always exists and is not unique)
\[
\omega=\alpha+\pa \beta\,,
\]
where $\alpha$ is a $2$-form with at most logarithmic pole along $D$, $\beta$ is a $(0,1)$-form with arbitrary order of poles along $D$, and $\pa$ is the holomorphic de Rham differential. The regularized integral \cite{Li:2020regularized} is defined to be
\[
\dashint_{\Sigma}\omega:= \int_{\Sigma}\alpha+ \int_{\pa \Sigma}\beta\,,
\]
where the right hand side is absolutely integrable and is independent of the choice of $\alpha, \beta$.

Such defined regularized integral has a few important properties  \cite{Li:2020regularized}.

\begin{prop}\label{prop-regularized-integral} Let $\Sigma$ be a compact Riemann surface without boundary, $D\subset \Sigma$ be a finite subset of points.  Let $\xi$ be a $(1,0)$-form and $\beta$ be a $(0,1)$-form, which are smooth away from $D$ and may admit holomorphic poles along $D$. Then
\[
\dashint_{\Sigma}\dbar \xi=-2\pi i \Res_\Sigma(\xi), \quad \dashint_{\Sigma}\pa \beta=0\,.
\]
Here $\Res_\Sigma(\xi)$ is the sum of local residues  defined by the limit  (note that $\xi$ may not be meromorphic)\footnote{The limit $\lim\limits_{\varepsilon\to 0}$ exists and does not depend on the choice of local  coordinate $z$  \cite[Definition 2.11]{Li:2020regularized}. }
\[
\Res_\Sigma(\xi)=\sum_{p\in D}\Res_{p}(\xi), \quad{where}\quad \Res_{p}(\xi)={1\over 2\pi i}\lim_{\varepsilon\to 0}\oint_{|z|=\varepsilon}\xi\,.
\]
Here $z$ is a local holomorphic coordinate around $p$ with $z(p)=0$.
\end{prop}

\begin{rem} Similar results hold when $\Sigma$ has a boundary  \cite[Theorem 2.14]{Li:2020regularized}.

\end{rem}

Assume from now on that $\Sigma$ has no boundary. The regularized integrals are extended to the configuration spaces of $\Sigma$. For $n\geq 2$, let $\Sigma^n$ be the $n$-th Cartesian product of $\Sigma$.  Let
$
\Delta_{ij}:= \{(p_1,\cdots, p_n)\in \Sigma^n| p_i=p_j\}
$
and $\Delta_{n}$ be the collection (called the \emph{big diagonal}) of all such diagonal divisors
\[
\Delta_{n}= \bigcup_{1\leq i\neq j\leq n} \Delta_{ij}\,.
\]
Let $\Omega$ be a $2n$-form on $\Sigma^n$ which is smooth away from $\Delta_{n}$ but may admit holomorphic poles of arbitrary order along $\Delta_{n}$. Then we can similarly define the regularized integral $
\dashint_{\Sigma^n}\Omega
$ via a decomposition of $\Omega$ into  logarithmic term and $\pa$-exact term \cite[Definition 2.38]{Li:2020regularized}. It  can be shown to be equal to the $n$-times iterated regularized integral on $\Sigma$
\[
\dashint_{\Sigma^n}\Omega= \dashint_\Sigma \dashint_{\Sigma}\cdots \dashint_{\Sigma}\Omega\,.
\]
In particular, this implies a Fubini-type theorem for regularized integrals \cite[Corollary 2.39]{Li:2020regularized}: the value of $\dashint_\Sigma \dashint_{\Sigma}\cdots \dashint_{\Sigma}\Omega$ does not depend on the choice of the ordering of regularized integrals on the factors of $\Sigma^{n}$.  This property will play a fundamental role in this paper.

\subsection{Preliminaries on functions on configuration spaces of elliptic curves}
\label{secpreminaries}

For any  point $\tau$ on the upper half-plane $\H$, let
\begin{equation*}
 E_{\tau}=\C/\Lambda_{\tau}\,,\quad \Lambda_{\tau}:=\Z\oplus \Z \tau
 \end{equation*}
 be the corresponding elliptic curve.
We fixed once and for all a
 linear holomorphic coordinate $z$ on the universal cover $\C$, and a volume form $\mathrm{vol}$ that
satisfies
 \begin{equation*}
 \int_{E_{\tau}}\mathrm{vol}=1\,,\quad
 \mathrm{vol}:= { \sqrt{-1} \over 2 \im \tau}dz\wedge d\bar z\,.
 \end{equation*}

 Throughout this work,
 constructions on the elliptic curve $E_{\tau}$ (and on the corresponding relative version
 $\mathcal{E}\rightarrow \H$) will be identified with their lifts
 to the universal cover.
 For example, meromorphic functions on $E_{\tau}$
 are identified with those on $\mathbb{C}$
 that are periodic under the translation action by $\Lambda_{\tau}$.
To perform integrations, we fix a basis  $\{A,B\}$ for $H_{1}(E_{\tau},\mathbb{Z})$. On the universal cover $\mathbb{C}\rightarrow E_{\tau}$, $A$ is represented by the segment $[\tau,\tau+1]$ and $B$ by the segment $[1,1+\tau]$. Such $A,B$ will be called the \emph{canonical representatives}.

\begin{figure}[h]\centering
	\begin{tikzpicture}[scale=1]

\draw[cyan!23,fill=cyan!23](0,0)to(4,0)to(5,3)to(1,3)to(0,0);
\draw (0,0) node [below left] {$0$} to (4,0) node [below] {$1$} to (5,3) node [above right] {$1+\tau$} to (1,3) node [above left] {$\tau$} to (0,0);

\draw[ultra thick,blue,->-=.5,>=stealth](1,3)to(5,3);

\node [above] at (3,3) {$A$};

\draw[ultra thick,red,->-=.5,>=stealth](4,0)to(5,3);

\node [right] at (4.5,1.5) {$B$};

	\end{tikzpicture}
\end{figure}

Let $E_{k}(\tau),k\geq 2$ be the usual Eisenstein series.
It is a standard fact that
\begin{equation}\label{eqndfinitionofY}
\widehat{E}_{2}:=E_{2}+{3\over \pi^2}\mathbold{Y}\,,\quad \boxed{ \mathbold{Y}:=-{\pi\over \mathrm{im}\,\tau}}\,
\end{equation}
 transforms as a modular form of weight two, while $E_{2}$ does not so.
 The former is called an almost-holomorphic modular form while the latter quasi-modular form \cite{Kaneko:1995}.
 Furthermore, the rings of quasi-modular forms
$\mathbb{C}[E_{4}, E_{6}, E_{2}]$ and almost-holomorphic modular forms $\mathbb{C}[E_{4}, E_{6}, \widehat{E}_{2}]\subseteq \mathbb{C}[E_{4}, E_{6}, E_{2}]\otimes \mathbb{C}[\mathbold{Y}] $
are isomorphic \cite{Kaneko:1995} via the so-called \emph{holomorphic limit} map
\begin{equation}\label{eqnhollimitmap}
\lim_{\mathbold{Y}=0}:~ \mathbb{C}[E_{4}, E_{6}][ \widehat{E}_{2}]\rightarrow  \mathbb{C}[E_{4}, E_{6}][ E_{2}]\,.
\end{equation}
We furthermore
introduce the following functions
\begin{eqnarray}\label{eqnWeierstrassellipticfunction}
&&{Z}(z):=\zeta(z)-{\pi^2\over 3}E_{2}(\tau)z\,,\quad
\widehat{Z}(z,\bar{z}):=\zeta(z)-{\pi^2\over 3}E_{2}(\tau)z+\mathbold{A}\,,\quad  \boxed{\mathbold{A}:=-{\pi\over \mathrm{im}\,\tau} (\overline{z}-z)}\,\nonumber\\
&&\widehat{P}(z):=-\partial_{z}\widehat{Z}=\wp(z)+{\pi^2\over 3}E_{2}(\tau)+\mathbold{Y}
\,,
\end{eqnarray}
where $\wp$, $\zeta$ are the Weierstrass elliptic functions given by
\begin{eqnarray*}
\wp(z)&=&{1\over z^2}+\sum_{(m,n)\in \mathbb{Z}\oplus \mathbb{Z}\setminus \{ (0,0)\}} \left( {1\over (z-(m\tau+n))^2}-{1\over (m\tau+n)^2}\right)\,,\\
\zeta(z)&=&{1\over z}+\sum_{(m,n)\in \mathbb{Z}\oplus \mathbb{Z}\setminus \{ (0,0)\}} \left( {1\over z-(m\tau+n)}-{1\over (m\tau+n)}+{z\over (m\tau+n)^2}\right)\,.
\end{eqnarray*}

Standard facts (see e.g., \cite{Silverman:2009}) from elliptic function theory tell that
\begin{equation}\label{eqnZtransform}
Z(z+1)=Z(z)\,,\quad Z(z+\tau)=Z(z)-2\pi i\,.
\end{equation}
Of crucial importance is the following addition formula for the Weierstrass $\zeta$-function
\begin{equation}\label{eqnadditionformula}
\zeta(z+w)=\zeta(z)+\zeta(w)+{1\over 2} {\wp'(z)-\wp'(w)\over \wp(z)-\wp(w)}\,,
\end{equation}
where $\wp'(z):=\partial_{z}\wp(z)$.
It implies that \begin{equation}\label{eqnapplyingadditionformula}
f(z_{i}-z_{k})=f(z_{i}-z_{j})+f(z_{j}-z_{k})+{1\over 2} {\wp'(z_{i}-z_{j})-\wp'(z_{j}-z_{k})\over \wp(z_{i}-z_{j})-\wp(z_{j}-z_{k})}\,\quad
\mathrm{for}~f=Z~\mathrm{or}~\widehat{Z}\,,
\end{equation}
for any triple of distinct values $(i,j,k)$.

\begin{rem}\label{remarkGreen}

The Green's function for the flat metric on the elliptic curve is given by
\begin{equation*}\label{eqnGreenfunction}
G(z,\bar{z})=\ln\left( e^{-2\pi {( \mathrm{im}\,z)^{2}\over \mathrm{im}\,\tau}} |{\vartheta_{({1\over 2}, {1\over 2})}(z) \over \eta(\tau)^{3}}|^{2}
\right)\,,\quad
\end{equation*}
where $\vartheta_{({1\over 2}, {1\over 2})}(z)$ is the Jacobi theta function with characteristic $({1\over 2},{1\over 2})$ and $\eta(\tau)$ is the Dedekind eta-function.
Then one can check that
\begin{equation*}\label{eqnpropagatorfunction}
\widehat{Z}=\partial_{z}G\,,
\quad
\widehat{P}=-\partial_{z}\widehat{Z}=-\partial_{z}^2 G\,.
\end{equation*}
In particular, the function
$\widehat{Z}$
is the Cauchy kernel for $\bar{\partial}$:
as distributions one has
\begin{equation*}\label{eqndistribution}
\bar{\partial}(\widehat{Z}dz)=\bar{\partial}\partial G=\mathrm{vol}-\delta\,,
\end{equation*}
where $\delta$ is the Dirac delta-distribution supported at the origin $z=0$ of the elliptic curve.
\end{rem}

For $n\geq 2$,
let $D_{n}\subseteq \mathbb{C}^{n}\times \H$ be the lift of the big diagonal, that is,
\[
D_{n}=\{(z_{1},\cdots ,z_{n},\tau)~|~\exists\, 1\leq i\neq j\leq n,\,\mathrm{such~ that}~ z_{i}-z_{j}\in \Lambda_{\tau}\}\,.
\]
Let $\mathcal{F}_{n}$ be the ring of functions on $\mathbb{C}^{n}\times \H$ that are smooth everywhere, except with possibly holomorphic
poles along $D_n$.
We are mostly interested in
   the following subspaces of $\mathcal{F}_{n}$.

   \begin{dfn}\label{dfnalmostelliptic}
For $n\geq 2$, define the subrings of $\mathcal{F}_{n}$ of
elliptic functions $  {\mathcal{J}}_{n}$, quasi-elliptic functions $  \widetilde{\mathcal{J}}_{n}$, almost-elliptic functions  $  \widehat{\mathcal{J}}_{n}$ to be
   \[
    {\mathcal{J}}_{n}:=\mathcal{F}_{n}\cap \mathbb{C}(\wp_{ij}, \wp'_{ij}, E_{4}, E_{6})\,,\quad
     \widetilde{\mathcal{J}}_{n}:= {\mathcal{J}}_{n}\,[Z_{ij}, E_{2}]\,,\quad
      \widehat{\mathcal{J}}_{n}:= {\mathcal{J}}_{n} \,[\widehat{Z}_{ij}, \widehat{E}_{2}]   \,,
    \]
 with the relations\footnote{Strictly speaking, there are further relations among the generators in the  ring   ${\mathcal{J}}_{n}$ such as the Weierstrass equations. }
  \eqref{eqnapplyingadditionformula} inherited from $\mathcal{F}_{n}$.
Hereafter $f_{ij}=f(z_{i}-z_{j})$, and
 by $R[f_{ij}],R(f_{ij})$ we mean the ring, fractional field generated by $f_{ij},1\leq  i\neq j\leq n$ over  $R$, respectively.

 For the $n=1$ case, we use the following convention
  \[
 \mathcal{J}_{1}=\mathbb{C}(\wp, \wp', E_{4}, E_{6})\,,\quad
  \widetilde{\mathcal{J}}_{1}= {\mathcal{J}}_{1}\,[Z, E_{2}]\,,\quad
   \widehat{\mathcal{J}}_{1}= {\mathcal{J}}_{1}\,[\widehat{Z}, \widehat{E}_{2}]\,.
 \]
 We put a weight grading $w$ on the above rings by the following assignment on the generators
 \[
 w(E_{k})=k\,,\quad
 w(\wp)=2\,,\quad
 w(\wp')=3\,,\quad
 w(Z)=w(\widehat{Z})=1\,.
 \]
 For later use we also put the convention $ w(z)=-1=-w(\partial_{z})$.
With this weight grading understood,
a notation such as $\mathcal{J}_{n,k}$ stands for the subset of $\mathcal{J}_{n}$ consisting of the weight-$k$ elements.
\end{dfn}

We shall frequently make use of the following relation without explicit mentioning
\[
   \widehat{\mathcal{J}}_{n}\subseteq  \widetilde{\mathcal{J}}_{n}\otimes \mathbb{C}[\mathbold{A}_{ij}, \mathbold{Y}]\,.
\]
That is, we
regard elements in $\widehat{\mathcal{J}}_{n}$
as polynomials in the generators $\mathbold{A}_{ij}, \mathbold{Y}$
with coefficients in $\widetilde{\mathcal{J}}_{n}$, where the generators $\mathbold{A}_{ij}$ are
 algebraically independent of $\mathbold{Y}$.

\begin{rem}
Functions in the graded rings $ {\mathcal{J}}_{n},     \widetilde{\mathcal{J}}_{n} ,    \widehat{\mathcal{J}}_{n}$
are  variants of Jacobi forms of index zero.
The weight grading $w$ introduced above  corresponds to the modular weight.
See \cite{Eichler:1985,  Libgober:2009, Dabholkar:2012, Goujard:2016counting}
for details.
\end{rem}

\subsection{Residue formulas for regularized integrals on elliptic curves}\label{secresidueformulasregularizedintegrals}

We next derive formulas that express regularized integrals on configuration spaces of $E$ in terms of residues.

\begin{lem}\label{lemregularizedintegralintermsofresidue}
Let $\Psi=\sum_k \Psi_{k}\widehat{Z}^{k}, \Psi_k\in {\mathcal{J}}_1$ be an almost-elliptic function. Then one has
\[\dashint_{E}\Psi\,\mathrm{vol}
=\sum_{k}{1\over k+1}\mathrm{Res} (\Psi_{k} \widehat{Z}^{k+1} dz )\in \mathbb{C}(E_{4},E_{6})[\widehat{E}_{2}]\,.
\]
Here $\Res$ is understood as in Proposition \ref{prop-regularized-integral}.
\end{lem}
\begin{proof}
 Observe that
\begin{equation}\label{eqnvolformintermsofZhat}
\mathrm{vol}={1\over 2\pi i}dz\wedge d\widehat{Z}\,.
\end{equation}
It follows from Proposition \ref{prop-regularized-integral} that
\[
\dashint_{E}\Psi\,\mathrm{vol}
=\sum_{k}{1\over 2\pi i }\dashint_{E}\Psi_{k} \widehat{Z}^{k} dz\wedge d\widehat{Z}
=-\sum_{k}{1\over k+1}{1\over 2\pi i}\dashint_{E} \dbar \bracket{\Psi_{k}  \widehat{Z}^{k+1}dz}
=\sum_{k}{1\over k+1}\mathrm{Res} (\Psi_{k} \widehat{Z}^{k+1} dz )\,.\]

By  \eqref{eqnWeierstrassellipticfunction}, one then obtains
\[\dashint_{E}\Psi\,\mathrm{vol}=\sum_{k}{1\over k+1}\mathrm{Res} \left(\Psi_{k}  \cdot (\zeta(z)-{\pi^2\widehat{E}_{2}\over 3}z-{\overline{z}\over \mathrm{im}\,\tau}) ^{k+1} dz \right)\,.
\]
According to \cite[Proposition 2.17]{Li:2020regularized} or the definition of $\Res_{p}(\xi)$ given in Proposition \ref{prop-regularized-integral}, one has
\begin{equation}\label{eqantiholomorphicresidue}
\Res_{z=0} (\overline{z}^{\ell}f)=0\,,\quad \ell\geq 1\,,
\end{equation}
for any meromorphic function $f$ in $z$.
Standard facts about Weierstrass elliptic functions tell that
the Laurent coefficients of $\wp(z), \wp'(z), {\zeta}(z) $ in $z$ are elements in $\mathbb{C}[E_{4},E_{6}]$.
Combining these, one arrives at the conclusion that $\dashint_{E}\Psi\,\mathrm{vol}\in \mathbb{C}(E_{4},E_{6})[\widehat{E}_{2}]$.
\end{proof}

\begin{rem} \label{remvolumeformforquasiness}
In \cite{Li:2020regularized}, instead of
 \eqref{eqnvolformintermsofZhat} we have used the identity
 \[
\mathrm{vol}=dz\wedge d\left( {i\over 2}{\bar{z}-z\over \mathrm{im}\tau}\right)
\]
which relates  regularized integrals to ordered $A$-cycle integrals and almost-holomorphic modular forms to quasi-modular forms (see also \cite{Douglas:1995conformal, Dijkgraaf:1997chiral}).
The latter expression for $\mathrm{vol}$ is natural from the Hodge theoretic viewpoint of quasi-modular forms
\cite{Urban:2014nearly, Ruan:2019}.
\end{rem}

Formally, Lemma \ref{lemregularizedintegralintermsofresidue}
reads that for an almost-elliptic function $f$ one has
\begin{equation}\label{eqnregularizedintegralintermsofresidue}
\dashint_{E} f\,=\mathrm{Res} \circ \partial_{\hat{Z}}^{-1}f\,,
\end{equation}
here $\partial_{\widehat{Z}}^{-1}$  indicates the anti-derivative with respect to $\widehat{Z}$. \\

Now let $\Phi\in {\mathcal{J}}_{n},n\geq 2 $ be an elliptic function. We will derive a residue formula for
\[
\dashint_{E_{[n]}}\Phi=\dashint_{E_{n}}\dashint_{E_{n-1}}\cdots \dashint_{E_{1}} \Phi\,.\]
Here $[n]=(1,2,\cdots,n)$, and $E_i$ indicates the $i$-th factor of their Cartesian product $E_{[n]}$ equipped with the coordinate $z_{i}$.
To emphasize the indexing, we shall sometimes write  $\mathcal{J}_{[n]},\widetilde{\mathcal{J}}_{[n]},
\widehat{\mathcal{J}}_{[n]}$ in place of $\mathcal{J}_{n},\widetilde{\mathcal{J}}_{n},\widehat{\mathcal{J}}_{n}$.

\begin{lem}\label{lempreservingalmostellipticity}
For any $\Psi\in \widehat{\mathcal{J}}_{[n],w}$ with $n\geq 2$, one has
\[
\dashint_{E_{i}}\Psi\in \widehat{\mathcal{J}}_{[n]-\{i\},w}\,.\]
In particular, one has $\dashint_{E_{[n]}}\Psi\in \mathbb{C}(E_{4},E_{6})[\widehat{E}_{2}]$.
\end{lem}
\begin{proof}
This is \cite[Proposition 3.15]{Li:2020regularized}. Here we give a different proof based on the residue formula in
Lemma \ref{lemregularizedintegralintermsofresidue}.
Pick any index not equal to $i$, say, $m$.
By applying the addition formula \eqref{eqnapplyingadditionformula} to $f=\widehat{Z}$, one can write
\[
\Psi=\sum_{k} \Psi_{k}\cdot (\widehat{Z}(z_{i}-z_{m}))^{k}\,,
\]
where $ \Psi_{k}\in\widehat{\mathcal{J}}_{[n], w-k}$ is meromorphic in $z_{i}$.
From Lemma \ref{lemregularizedintegralintermsofresidue}, one sees that
\[
\dashint_{E_{i}}\Psi=\sum_{k}
{1\over k+1}\mathrm{Res} \left(\Psi_{k} \cdot (\widehat{Z}(z_{i}-z_{m}))^{k+1} dz_{i} \right)
=\sum_{k}
{1\over k+1}\sum_{j} \Res_{z_{i}=z_{j}} \left(\Psi_{k}\cdot ( \widehat{Z}(z_{i}-z_{m}))^{k+1} dz_{i} \right)\,.
\]
If $n=2$, then $j=m$. Otherwise we apply the addition formula  \eqref{eqnapplyingadditionformula} to the triple of indices $(i,j,m)$ with $f=\widehat{Z}$. In either case, it suffices to show that for any
$ \psi_{\ell}\in\widehat{\mathcal{J}}_{[n], w-\ell}$ meromorphic in $z_{i}$, one has
\[
\Res_{z_{i}=z_{j}} \left(\psi_{\ell} \cdot (\widehat{Z}(z_{i}-z_{j}))^{\ell+1} dz_{i} \right)
\in \widehat{\mathcal{J}}_{[n]-\{i\},w}\,.\]
By \eqref{eqnWeierstrassellipticfunction} and \eqref{eqantiholomorphicresidue}, this is equivalent to showing
\[
\Res_{z_{i}=z_{j}} \left(\psi_{\ell} \cdot ({\zeta}(z_{i}-z_{j})-{\pi^{2}\widehat{E}_{2}\over 3}(z_{i}-z_{j}))^{\ell+1} dz_{i} \right)
\in \widehat{\mathcal{J}}_{[n]-\{i\},w}\,.
\]

According to Definition \ref{dfnalmostelliptic},
the function $\psi_{\ell}$ is represented as a polynomial in $\widehat{Z}(z_{a}-z_{b}),a,b\neq i$, with coefficients
being rational functions in $E_{4},E_{6},\wp(z_{i}-z_{k}), \wp'(z_{i}-z_{k}),\wp(z_{a}-z_{b}), \wp'(z_{a}-z_{b}), k,a,b\neq i $.
Using the addition formula \eqref{eqnadditionformula}, one can see that thel
Laurent coefficients of $\wp(z_{i}-z_{k}), \wp'(z_{i}-z_{k}), {\zeta}(z_{i}-z_{j}) $ in $z_{i}-z_{j}$ are elements in $\mathcal{J}_{[n]-\{i\}}$.
This then gives the desired  almost-ellipticity above.
The claim about the weight grading follows from the assignment that $w(z)=-1=-w(\partial_z)$.
The claim $\dashint_{E_{[n]}}\Psi\in \mathbb{C}(E_{4},E_{6})[\widehat{E}_{2}]$ follows by iteration and Lemma  \ref{lemregularizedintegralintermsofresidue}.
\end{proof}

Applying
Lemma \ref{lempreservingalmostellipticity} and
Lemma \ref{lemregularizedintegralintermsofresidue} to $\dashint_{E_{n-1}}\dashint_{E_{[n-2]}}\Psi, \Psi\in \widehat{\mathcal{J}}_{[n]}$, we see that the last integration
$\dashint_{E_{n}}$ in $\dashint_{E_{[n]}}$ is actually trivial.
For this reason, we can use $z_{n}\in E_{n}$ as a reference in all constructions.

\begin{dfn}\label{dfnWdefinition}
Let $z_i$ be the linear coordinate on $E_i$ and $z=(z_{1},\cdots, z_{n})$.
For any function $f$ on $\mathbb{C}$ and any $1\leq a \neq b<n$, set
\begin{eqnarray*}
&&z_{ab}=z_{a}-z_{b}\,,\quad
f_{a}(z)=f(z_{a}-z_{n})\,,\quad
f_{a,b}(z)=f(z_{a}-z_{n})-f(z_{b}-z_{n})=f_{a}(z)-f_{b}(z)\,.
\end{eqnarray*}
We also use the convention that $f_{a,n}(z)=f_{a}(z)$.
\end{dfn}
In what follows, we shall frequently take $f=Z, \widehat{Z},\mathbold{A}$.
For example, $\widehat{Z}_{a}(z)=\widehat{Z}(z_{a}-z_{n}),\widehat{Z}_{a,b}(z)=\widehat{Z}_{a}(z)-\widehat{Z}_{b}(z)$.\\

To proceed, we  introduce a few more notations on residue operators.
 \begin{dfn}\label{dfnresidueoperatorcompositions}  Denote
\[R^{(a)}_{b}=
\mathrm{Res}_{z_{a}=z_{b}}\,,
\]
that is, $R^{(a)}_{b}(f):= \Res_{z_a=z_b} (f dz_a)$. The following convention will be used when we compose residue operations
\[
R^{(i)}_{i}=0\,,\quad\quad
\cdots R^{(i_{k})}_{j_{k}}\cdots R^{(i_{2})}_{j_{2}} R^{(i_{1})}_{j_{1}}=0\,,
\quad
\mathrm{if}\,\quad \exists\, \,k_{1}< k_{2}\,\,\mathrm{such~that~}\,\,
i_{k_{2}}= i_{k_{1}}\,\,\mathrm{or}\,\, j_{k_{2}}= i_{k_{1}}\,.
\]
\end{dfn}

\begin{prop}\label{propresidueformulaforregularizedintegral}
Let $\Phi\in \mathcal{J}_{n},n\geq 2$ be an elliptic function. Then one has
 \begin{equation*}
 \dashint_{E_{[n]}}\Phi=\dashint_{E_{n}}
\sum_{\substack{\mathbold{r}=(r_{1},\cdots, r_{n-1})\\
r_{1}>1,\, r_{2}>2,\,\cdots, \, r_{n-1}=n}} R^{(n-1)}_{r_{n-1}}\circ \cdots\circ R^{(1)}_{r_{1}}
(\Phi F_{\mathbold{r}})\,,
\end{equation*}
where the function $F_{\mathbold{r}}$ is given by the following formal integral formula
\begin{equation}\label{eqnFrformula}
F_{\mathbold{r}}=\int_{0}^{\widehat{Z}_{n-1}+0} dx_{n-1}\cdots  \int_{0}^{\widehat{Z}_{k,r_{k}}+ x_{r_{k}}}dx_{k}\cdots \int_{0}^{\widehat{Z}_{1,r_{1}}+
x_{r_{1}}}dx_{1}\,.
\end{equation}
Here the quantities $\widehat{Z}_{k,r_{k}},k=1,2,\cdots , n-1$ (as defined in Definition \ref{dfnWdefinition}) regarded as integration constants, and the convention
$x_{n}=0$ is understood.
\end{prop}
\begin{proof}
In the $\dashint_{E_{m}}, m\geq 2$ integration, the integrand
$\dashint_{E_{m-1}}\cdots \dashint_{E_{1}}\Phi$,  regarded as a function in $z_{m}$,
has only possible poles located at
$z_{m}=z_{m+1},\cdots, z_{n}$ according to \cite[Proposition 2.26]{Li:2020regularized}. The proof follows by iterating Lemma \ref{lemregularizedintegralintermsofresidue} as follows.

From  Lemma \ref{lemregularizedintegralintermsofresidue} we have
\begin{eqnarray*}
 \dashint_{E_{1}} \Phi
& =&\sum_{r_{1}>1}
R^{(1)}_{r_{1}}(\widehat{Z}_{1}
 \Phi)
 =\sum_{r_{1}>1}\bracket{
R^{(1)}_{r_{1}}( \widehat{Z}_{1,r_{1}}
  \Phi)
  +
 R^{(1)}_{r_{1}}( \widehat{Z}_{r_{1}}
  \Phi)}\,.
\end{eqnarray*}
Similar to the reasoning in Lemma \ref{lempreservingalmostellipticity} using \eqref{eqantiholomorphicresidue},
we see that $R^{(1)}_{r_{1}}(\widehat{Z}_{1,r_{1}}\Phi)$ is meromorphic and thus an
elliptic function in $\mathcal{J}_{[n]-\{1\}}$.
On the other hand,  $R^{(1)}_{r_{1}} (\widehat{Z}_{r_{1}}\Phi)=\widehat{Z}_{r_{1}} R^{(1)}_{r_{1}} (\Phi)$ is
 almost-elliptic, where the anti-holomorphic dependence is encoded explicitly via $\widehat{Z}_{r_{1}}$.
Consider the
remaining iterated regularized integrals in
$\dashint_{E_{[n]}}\Phi=\dashint_{E_{n}}\dashint_{E_{n-1}}\cdots \dashint_{E_{1}}\Phi$.
The quantity $\widehat{Z}_{r_{1}}$ is treated as an integration constant when applying $\dashint_{E_i}$ for $i< r_1$, until for the $\dashint_{E_{r_{1}}}$-integration
in which
 we apply Lemma \ref{lemregularizedintegralintermsofresidue}.  In fact, for any elliptic function $\varphi_{k}$ one has
\[
\dashint_{E_{i}}  \varphi_k \widehat{Z}_{i}^k= \sum_{r_i>i} R^{(i)}_{r_i}\bracket{\varphi_k {\widehat{Z}_{i}^{k+1}\over k+1}}= \sum_{r_i>i} R^{(i)}_{r_i}\bracket{\varphi_k {(\widehat{Z}_{i ,r_i}+\widehat{Z}_{r_i})^{k+1}\over k+1}}\,.
\]
Now using
 Lemma \ref{lempreservingalmostellipticity} and the addition formula \eqref{eqnapplyingadditionformula} to the consecutive iterated integrals whenever necessary,
the desired formula of $F_\mathbold{r}$  follows by iterating this process.

\end{proof}

\begin{rem}\label{remindependenceofantiderivatives}
Let $D_{i}=\partial_{\widehat{Z}_{i}}$ be the derivative in the generator $\widehat{Z}_{i}$ when acting on almost-elliptic functions.
A different choice of the anti-derivative $\partial_{\widehat{Z}_{i}}^{-1}(\widehat{Z}_{i}^k)$ in the above proof is differed by addition by an elliptic function in the remaining variables.
By applying the global residue theorem to the consecutive summation of residues,
this difference results in zero in the overall regularized integral.
\end{rem}

Using the property that the regularized integral $\dashint_{E_{[n]}}$ is independent of the ordering for the iterated regularized integrals, we can obtain the following alternative formula
that is analogous to \cite[Proposition 9, Proposition 10]{Oberdieck:2018} for ordered $A$-cycle integrals.
See Section \ref{secAcycleintegrals}
 for further discussions on their relations.

\begin{cor}\label{corinductiveresidueformulaforergularizedintegral}
Let $\Phi$ be an elliptic function in $\mathcal{J}_{n}$. Then one has
\begin{equation*}
\dashint_{E_{[n]}}\Phi=\dashint_{E_{n}}
\sum_{I: \,I=(i_{1},\cdots, i_{m})} \dashint_{E_{[n-1]-I}} R^{(i_{m})}_{n}\circ \cdots\circ R^{(1)}_{i_{2}}
(\Phi F_{I})\,,
\end{equation*}
where the summation is over all non-recurring sequences $I=(i_{1},\cdots ,i_{m})$ with
$i_{1}=1$, with length $m$ satisfying $1\leq m=|I|\leq n-1$, and with $i_{k}\neq n$ for $k=1,2,\cdots, m$.
The function $F_{I}$ can be chosen to be
\begin{equation}\label{eqnFIintegralformula}
\int_{0}^{\widehat{Z}_{i_{m}}+0} dx_{i_{m}}\cdots  \int_{0}^{\widehat{Z}_{i_{k},i_{k+1}}+ x_{i_{k}}}dx_{i_{k}}\cdots \int_{0}^{\widehat{Z}_{i_{1},i_{2}}+ x_{i_{2}}}dx_{i_{1}}\,,
\end{equation}
or alternatively chosen to be
\begin{equation}\label{eqnFIsimpleintegralformula}
{\widehat{Z}_{i_{1}}^{m}\over m!}\,.
\end{equation}
\end{cor}

\begin{proof}
The proof is similar to that of
Proposition \ref{propresidueformulaforregularizedintegral}. However,
instead of performing the iterated integration in the prescribed ordering $\dashint_{E_{n}}\dashint_{E_{n-1}}\cdots\dashint_{E_{1}}$ for integration, we apply the independence on the ordering and proceed as follows. By Lemma \ref{lemregularizedintegralintermsofresidue},  one has
\[\dashint_{E_{[n]}}\Phi=\dashint_{E_{n}}\cdots \dashint_{E_{2}}\dashint_{E_{1}}\Phi
=\dashint_{E_{n}}\cdots \dashint_{E_{2}}
\sum_{i_{2}:\,i_{2}\neq 1}R^{(1)}_{i_{2}}(\Phi \widehat{Z}_{1})\,.\]
If $i_{2}=n$, then
the above integrand $R^{(1)}_{i_{2}}(\Phi \widehat{Z}_{1})$ is meromorphic, and the integral is
\[
\dashint_{E_{n}}\dashint_{E_{[n-1]-\{1\}}} R^{(1)}_{n}(\Phi \widehat{Z}_{1})\,.
\]
If $i_{2}\neq n$, by using the independence on the integration order, the  corresponding integral is
\[
\sum_{i_{2}:\,i_{2}\neq n}\dashint_{E_{n}}\dashint_{E_{[n-1]-\{1,i_{2}\}}}\dashint_{E_{i_{2}}}
R^{(1)}_{i_{2}}(\Phi \widehat{Z}_{1})
=
\sum_{i_{2}:\,i_{2}\neq n}\sum_{i_{3}}\dashint_{E_{n}}\dashint_{E_{[n-1]-\{1,i_{2}\}}} R^{(i_{2})}_{i_{3}}\circ
R^{(1)}_{i_{2}}(\Phi\cdot  \widehat{Z}_{1,i_{2}} \widehat{Z}_{i_{2}}+ \Phi\cdot  {1\over 2}\widehat{Z}^{2}_{i_{2}})\,.
\]
If $i_{3}=n$, then the above integrand is meromorphic, otherwise one continues.
Iterating this procedure yields the desired formula of
$F_{I}$ given in \eqref{eqnFIintegralformula}.

As  explained in Remark \ref{remindependenceofantiderivatives},
different choices of the anti-derivative in any step results in zero in the overall regularized integral by the global residue theorem.
Now for $m\geq 1$, one has
\begin{eqnarray*}
\partial_{\widehat{Z}_{i_{m}}} \left({\widehat{Z}_{i_{1}}^{m}\over m!}
\right)
&=&
\partial_{\widehat{Z}_{i_{m}}} \left( {(\widehat{Z}_{i_{1},i_{2}}+\widehat{Z}_{i_{2},i_{3}}+\cdots + \widehat{Z}_{i_{m-1}, i_{m}}+\widehat{Z}_{i_{m}})^{m}\over m!}\right)\\
&=& {(\widehat{Z}_{i_{1},i_{2}}+\widehat{Z}_{i_{2},i_{3}}+\cdots + \widehat{Z}_{i_{m-1}})^{m-1}\over (m-1)!}
={\widehat{Z}_{i_{1}}^{m-1}\over (m-1)!}\,.
\end{eqnarray*}
It follows that \eqref{eqnFIsimpleintegralformula} also provides a coherent system of choices $\{F_{I}\}_{I}$ of anti-derivatives.\footnote{In fact,  $F_{I}$ in this system of choices also admits a formal integral formula
\[
\int_{-\widehat{Z}_{i_{1},i_{m}}}^{\widehat{Z}_{i_{m}}+0} dx_{i_{m}}\cdots  \int_{-\widehat{Z}_{i_{1},i_{k}}}^{\widehat{Z}_{i_{k},i_{k+1}}+ x_{i_{k+1}}}dx_{i_{k}}\cdots
\int_{-\widehat{Z}_{i_{1},i_{2}}}^{\widehat{Z}_{i_{2},i_{3}}+ x_{i_{3}}}dx_{i_{2}}
\int_{0}^{\widehat{Z}_{i_{1},i_{2}}+ x_{i_{2}}}dx_{i_{1}}\,.
\]
}
The proof is now complete.

\end{proof}

\section{Holomorphic anomaly equations}
\label{secanomalyequations}

Lemma \ref{lempreservingalmostellipticity} tells that
for any $\Psi\in \widehat{\mathcal{J}}_{n}, n\geq 2$ the regularized integral $\dashint_{E_{[n]}}\Psi$
is a function in  $\mathbb{C}(E_{4}, E_{6})[ \widehat{E}_{2}]$
and thus is a
polynomial in $\widehat{E}_{2}$ with coefficient being elements in $\mathbb{C}(E_{4}, E_{6})$. We can treat it as a polynomial in $\mathbold{Y}$
with coefficients being elements in $\mathbb{C}(E_{4}, E_{6}) [E_{2}]$.
The non-holomorphic dependence introduced by the regularized integration operator is highly structured, as we now discuss.

\subsection{Holomorphic anomaly of the regularized integration operator}

The anti-holomorphic dependence of the regularized integral $\dashint_{E_{[n]}}\Psi$ is captured by
 the following holomorphic anomaly equation.
\begin{thm}\label{thmHAE}
Let $\Psi\in \widehat{\mathcal{J}}_{n}, n\geq 2$ be an almost-elliptic function, then one has
\begin{equation}\label{eqnHAE}
\partial_{\mathbold{Y}}\dashint_{E_{[n]}}\Psi
=\dashint_{E_{[n]}} \partial_{\mathbold{Y}}\Psi -\sum_{a,b:\,a< b}
\dashint_{E_{[n]-\{a\}}} R^{(a)}_{b}( (z_{a}-z_{b})\Psi)\,.
\end{equation}
Here in computing $\partial_{\mathbold{Y}}\Psi$, $\Psi$ is regarded as an element in $\widetilde{\mathcal{J}}_{n}\otimes \mathbb{C}[\mathbold{A}_{ij},\mathbold{Y}]$
with the convention that $\partial_{\mathbold{Y}}\mathbold{A}_{ij}=0$.
\end{thm}
This equation should be compared with
 the one
  \cite[Theorem 7 (3)]{Oberdieck:2018} satisfied by ordered $A$-cycle integrals. See Section \ref{secAcycleintegrals}
 for further discussions.

We start by proving the following lemma.

\begin{lem}\label{lempartialregularizedintegral}
Let $\Psi$ be an almost-elliptic function in $\widehat{\mathcal{J}}_{n}$, with $ n\geq 2$.
Then one has
\begin{equation*}\label{eqnpartialregularizedintegral}
\partial_{\mathbold{Y}}\dashint_{E_{1}}\Psi=\dashint_{E_{1}}  \partial_{\mathbold{Y}} \Psi -\sum_{r}R^{(1)}_{r}( (z_1-z_r) \Psi)\,.
\end{equation*}
\end{lem}

\begin{proof}
By the linearity of the operators on both sides and by applying the addition formula \eqref{eqnapplyingadditionformula} as in Lemma \ref{lempreservingalmostellipticity},
it suffices to prove the statement for $\Psi=\widehat{Z}_1^{k}\varphi$, where
 $\varphi\in \mathcal{J}_{[n]-\{1\}}\otimes \mathbb{C}[\widehat{E}_{2}]$.
Applying Lemma \ref{lemregularizedintegralintermsofresidue}, one has
\begin{eqnarray*}
\partial_{\mathbold{Y}}\dashint_{E_1}\widehat{Z}_1^{k}\varphi&=&
\partial_{\mathbold{Y}} \sum_{r}R^{(1)}_{r}( {\widehat{Z}_{1}^{k+1}\over k+1}\varphi )\\
&=&\sum_{r} R^{(1)}_{r}( {\widehat{Z}_{1}^{k+1}\over k+1} \cdot \partial_{\mathbold{Y}} \varphi)
+\sum_{r}  R^{(1)}_{r}((\overline{z}_{1r}-z_{1r})\cdot \partial_{\widehat{Z}_{1,r}}   {(\widehat{Z}_{1,r}+\widehat{Z}_{r})^{k+1}\over k+1} \cdot \varphi)\\
&=&\dashint_{E_1}  \partial_{\mathbold{Y}}(\widehat{Z}_1^{k}\varphi )
-\sum_{r}R^{(1)}_{r}(z_{1r}\widehat{Z}_1^{k}\varphi )\,.
\end{eqnarray*}
Note that the operator
$R^{(1)}_{r}(\widehat{Z}_{1}^{k+1}-)$ can bring in non-trivial dependence in $\mathbold{Y}$ in the way displayed in the second term of the third expression above.
 In the last equality we have used
 the fact that $\partial_{\mathbold{Y}}\mathbold{A}_{ij}=0$,
 and
  $R^{(1)}_{r}(\overline{z}_{1r}-)=0$ according to \eqref{eqantiholomorphicresidue}.
\end{proof}

\begin{lem}\label{lemcommutatorresidue}
Let $\Psi$ be an almost-elliptic function in $\widehat{\mathcal{J}}_{n}$ with $ n\geq 3$.
Then one has the following identities of operators acting on $\Psi$
\begin{equation}\label{eqncommutatorofR}
R^{(a)}_{c}R^{(b)}_{c}-R^{(b)}_{c}R^{(a)}_{c}=R^{(b)}_{c}R^{(a)}_{b}
=-R^{(a)}_{c}R^{(b)}_{a}\,,\quad
\forall \,1\leq a,b,c\leq n\,,
\end{equation}
and
\begin{equation}\label{eqntrivialcommutatorofR}
R^{(c)}_{d}R^{(a)}_{b}=R^{(a)}_{b}R^{(c)}_{d}\,, \quad \mathrm{if}\,\quad \{a,b\}\cap \{c,d\}=\emptyset\,.
\end{equation}

\end{lem}

\begin{proof}
 The identities \eqref{eqncommutatorofR} and \eqref{eqntrivialcommutatorofR} for quasi-elliptic functions in $\widetilde{\mathcal{J}}_{n}$ are
proved  \cite[Lemma 21]{Oberdieck:2018} by a direct local calculation on monomials.
To be a little more detailed,  using the meromorphicity of quasi-elliptic functions along the diagonals and the linearity of the residue operations, it suffices to check the statement
for the monomial $z_{ac}^{m} z_{ab}^{n}, m,n\in \mathbb{Z}$.
The statement is not trivially true only when $m\leq 0, n\leq 0$
in which case the validity follows from the elementary identity
$${n\choose n+1}\delta_{n+m+1,-1}\delta_{\{n+1\geq 0\}}={m\choose m+1}\delta_{n+m+1,-1}\delta_{\{m+1\geq 0\}}\,.$$
For almost-elliptic functions, using the polynomiality in $\widehat{Z}$,
to prove \eqref{eqncommutatorofR} it suffices to check the statement for the monomial $z_{ac}^{m} z_{ab}^{n}\overline{z}_{ac}^{r}\overline{z}_{ab}^{s},m,n\in \mathbb{Z}, r,s\in \mathbb{N}$.
The result \eqref{eqantiholomorphicresidue} tells that the statement is non-trivial only when $r=s=0$.
This then reduces to the local calculation just mentioned.
The identity \eqref{eqntrivialcommutatorofR} is also proved by the same type of  local calculation.
\end{proof}
\begin{proof}[Proof of Theorem \ref{thmHAE}]
 We now prove the holomorphic anomaly equation by induction on  $n$.
The initial case with $n=2$ follows from Lemma \ref{lempartialregularizedintegral} and the triviality of the last one in the iterated regularized integrations.
Assuming the statement is true for $n-1,n\geq 3$, we
now prove the statement for $n$.
By the induction hypothesis, we have
\[\partial_{\mathbold{Y}}\dashint_{E_{[n]}}\Psi=
\dashint_{E_{[n]-\{1\}}}
\partial_{\mathbold{Y}}\dashint_{E_1}\Psi
-\dashint_{E_{[n]-\{1\}}}\sum_{2\leq a<b}R^{(a)}_{b} z_{ab}\dashint_{E_{1}}\Psi
\,.
\]
Applying  Lemma \ref{lempartialregularizedintegral}, it remains to prove
\begin{equation*}\label{eqnRabwabint1}
\sum_{2\leq a<b}R^{(a)}_{b}z_{ab}\dashint_{E_{1}}\Psi=\dashint_{E_{1}}\sum_{2\leq a<b}R^{(a)}_{b}z_{ab}\Psi\,.
\end{equation*}
By Lemma \ref{lemregularizedintegralintermsofresidue} and  \eqref{eqntrivialcommutatorofR} in Lemma \ref{lemcommutatorresidue},
it is enough to show that
\begin{equation}\label{eqnRabwabintsimplified}
R^{(a)}_{b}z_{ab} \left( R^{(1)}_{a} [\partial_{\widehat{Z}_{1}}^{-1}\Psi]
+R^{(1)}_{b} [\partial_{\widehat{Z}_{1}}^{-1}\Psi]\right)
= R^{(1)}_{b} [\partial_{\widehat{Z}_{1}}^{-1} R^{(a)}_{b}z_{ab} \Psi]\,.
\end{equation}
Here
as in the proof of Lemma \ref{lempartialregularizedintegral},
we have applied the addition formula  \eqref{eqnapplyingadditionformula}  to express $\widehat{Z}(z_1-z_i),i\neq n$ involved in $\Psi$
in terms of $\widehat{Z}_{1}$ and elliptic functions if needed.
Recall that $\partial_{\widehat{Z}_{1}}^{-1}$ arises from $\partial_{\overline{z}_{1}}^{-1}$.
Since
 by assumption $a,b\neq 1$, a local computation tells that for the right hand side of \eqref{eqnRabwabintsimplified}
\[ R^{(1)}_{b} [\partial_{\widehat{Z}_{1}}^{-1}R^{(a)}_{b}z_{ab} \Psi]
= R^{(1)}_{b} [R^{(a)}_{b}z_{ab}\partial_{\widehat{Z}_{1}}^{-1}\Psi]\,.
\]
The desired identity \eqref{eqnRabwabintsimplified} then follows from
\eqref{eqncommutatorofR} in Lemma \ref{lemcommutatorresidue}.
This finishes the inductive proof for the $n$ case.

\end{proof}

\subsection{Reformulations and interpretations}

The holomorphic anomaly equation \eqref{eqnHAE} in Theorem \ref{thmHAE}  now offers a practical way,
besides the residue formulas in
Proposition \ref{propresidueformulaforregularizedintegral} and Corollary \ref{corinductiveresidueformulaforergularizedintegral},  to recursively compute regularized integrals.
Moreover,
it exposes some interesting structures of the regularized integration operator,
as we now explain.

\subsubsection{Modular anomaly}
When acting on almost-elliptic functions,  \eqref{eqnHAE}
can be expressed as an identity between operators
\begin{equation}\label{eqnHAEreformulated}
\partial_{\mathbold{Y}}\,(\dashint_{E_{[n]}})=-{1\over 2}\sum_{a,b:\,b\neq a}\dashint_{E_{[n]-a}}R^{(a)}_{b}(z_{ab}-)\,\quad
\mathrm{when~ acting ~on}\,~\widehat{\mathcal{J}}_{n}\,,
\end{equation}
where the left hand side is the usual action on operators, that is, $\partial_{\mathbold{Y}}\,(\dashint_{E_{[n]}}):=\partial_{\mathbold{Y}}\circ \dashint_{E_{[n]}}-\dashint_{E_{[n]}}\circ\partial_{\mathbold{Y}}$.

To proceed, we define the following operators similar to Definition \ref{dfnresidueoperatorcompositions}.
\begin{dfn} Define
\begin{eqnarray*}
&S^{(a)}_{b}=R^{(a)}_{b}(z_{ab}-)\,,\quad
S^{(a)}=\sum_{ r:\, r\neq a}R^{(a)}_{r}(z_{ar}-)\,,\quad
S=\sum_{a}S^{(a)}\,,\\
&\dashint S^{k}=\sum_{i_{1}, j_{1}}\sum_{i_{2},j_{2}}\cdots \sum_{i_{k},j_{k}}\dashint_{E_{[n]-\{i_1,i_{2}, \cdots,\, i_{k}\}}}
S^{(i_{k})}_{j_{k}}\cdots S^{(i_{2})}_{j_{2}}S^{(i_{1})}_{j_{1}}\,.
\end{eqnarray*}
\end{dfn}
\begin{lem}\label{lemSpreserve}
The operator
$S^{(a)}_{b}$ satisfies
\[S^{(a)}_{b}:
{\mathcal{J}}_{n,\,k}\rightarrow
{\mathcal{J}}_{n-1,\,k-2}\,,\quad
\widetilde{\mathcal{J}}_{n,\,k}\rightarrow
\widetilde{\mathcal{J}}_{n-1,\,k-2}\,,\quad
\widehat{\mathcal{J}}_{n,k}\rightarrow
\widehat{\mathcal{J}}_{n-1,\,k-2}\,,\quad n\geq 2\,.
\]

\end{lem}

\begin{proof}
The proof is similar to that of Lemma \ref{lempreservingalmostellipticity}.
\end{proof}

Now
\eqref{eqnHAEreformulated}
  can be further abbreviated as follows.
  \begin{cor}\label{corHAEabbreviation}
As operators on almost-elliptic functions in $\widehat{\mathcal{J}}_{n}$, the following holds
\begin{equation}\label{eqnHAEabbreviation}
\partial_{\mathbold{Y}}\,(\dashint_{E_{[n]}})=-{1\over 2 }\sum_{a}\dashint_{E_{[n]-a}}S^{(a)}=-{1\over 2}\dashint S\,.
\end{equation}
\end{cor}
Consequently, one has the following convenient closed-form expression for the regularized integration operator on elliptic functions.
\begin{thm}\label{thmHAEBCOV}
As operators on elliptic functions in ${\mathcal{J}}_{n}$, the following holds
\begin{equation}\label{eqnHAEintegrated}
\dashint_{E_{[n]}}=-{1\over 2}\mathbold{Y} \dashint S
-({1\over 2})^{2}{1\over 2!}\mathbold{Y}^{2} \dashint SS
-({1\over 2})^3{1\over 3! }\mathbold{Y}^{3} \dashint SSS+\cdots
+{1\over n!}\sum_{\sigma\in\mathfrak{S}_{n}}\int_{A_{\sigma([n])}}\,.
\end{equation}
\end{thm}

\begin{proof}
This follows by integrating \eqref{eqnHAEabbreviation} with respect to $\mathbold{Y}$,
 Lemma \ref{lemSpreserve} that $S$ preserves ellipticity, and \eqref{eqnaveraging} proved in \cite[Theorem 3.4]{Li:2020regularized}.
\end{proof}

The formulation of holomorphic anomaly equation in Theorem \ref{thmHAEBCOV}
resembles the form
of those in the physics literature
 \cite{Bershadsky:1993cx, Dijkgraaf:1997chiral}, and
 offers a mathematical formulation of contact term singularities systematically studied by Dijkgraaf \cite{Dijkgraaf:1997chiral}  (see also  \cite{rudd1994string, Douglas:1995conformal})
 via the $S$-operators.\\

We can also  reformulate Theorem \ref{thmHAEBCOV} using
 the generating series technique.
 For any formal series $F(T)=\sum_{n\geq 0}a_{n} {T^{n}} $ in $T$ and a sequence of non-commutative operators $O=\{O_{k}\}_{k\in\mathbb{N}}$, we define
 \[
 F(O)=\sum_{n\geq 0}{a_{n} \over n!}\sum_{\sigma\in \mathfrak{S}_{n}} O_{\sigma(n)}\circ O_{\sigma(n-1)}\cdots\circ O_{\sigma(1)}\,.
 \]
Then Theorem \ref{thmHAEBCOV} can be rewritten as
\begin{equation*}
e^{{1\over \hbar}\dashint}
=e^{{1\over \hbar}\dashint}\circ (\mathbold{1}- e^{{1\over 2}{1\over \hbar}\mathbold{Y}S})+e^{{1\over \hbar}\int_{A}}\,,
\end{equation*}
as formal generating series in ${1 / \hbar}$ of graded non-commutative operators acting on elliptic functions.
Here any operator appearing as a coefficient in $1/\hbar^{n}$ is understood to act on
elliptic functions in $\mathcal{J}_{n}$.
Furthermore, \eqref{eqnHAEabbreviation}
gives
\begin{equation*}\label{eqnHAEMCform}
\left(\partial_{\mathbold{Y}}+\circ {1\over 2}{1\over \hbar} S \right)e^{{1\over \hbar}\dashint}=0\,.
\end{equation*}

\subsubsection{Elliptic anomaly}

When acting on $\mathcal{J}_{n}$,
the regularized integral operators $\dashint_{E_{I}}, I\subseteq [n]$
 introduces not only non-holomorphicity
 $\tau$, but also
non-holomorphicity in $z$.
The latter enters through polynomials in $\mathbold{A}_{ij}$, according to
Lemma \ref{lempreservingalmostellipticity} and Definition \ref{dfnalmostelliptic}.
One then has the following result that is parallel to the modular anomaly equation in
Corollary \ref{corHAEabbreviation}.
\begin{cor}\label{corellipticanomaly}
As operators on elliptic functions in $\mathcal{J}_{n}$, one has
\[\partial_{\mathbold{A}_{b}}(\dashint_{E_{a}})
=
R^{(a)}_{b}\,.\]
Here  $\mathbold{A}_{b}=\mathbold{A}(z_{b}-z_{n})$ by Definition \ref{dfnWdefinition}, and the convention $\partial_{\mathbold{A}_{b}}\mathbold{Y}=0$ is understood.
\end{cor}
\begin{proof}
Let $\Phi$ be an elliptic function.
Recall from the proof of Proposition \ref{propresidueformulaforregularizedintegral}
that
\[\dashint_{E_{a}}\Phi
=\sum_{r:\, r\neq a}R^{(a)}_{r} (\widehat{Z}_{a}\Phi)
=\sum_{r: \, r\neq a}R^{(a)}_{r} (\widehat{Z}_{a,r}\Phi)+\sum_{r: \, r\neq a}R^{(a)}_{r} (\widehat{Z}_{r}\Phi)
\,.\]
According to \eqref{eqantiholomorphicresidue}, the first term in the last expression is meromorphic in $z$.
It follows that
\begin{eqnarray*}
\partial_{\mathbold{A}_{b}}
\dashint_{E_{a}}\Phi
= \partial_{\mathbold{A}_{b}}(\widehat{Z}_{b})\cdot R^{(a)}_{b} (\Phi)
=R^{(a)}_{b} (\Phi)\,.
\end{eqnarray*}
\end{proof}
More generally, using  \eqref{eqnFrformula} in Proposition \ref{propresidueformulaforregularizedintegral}, one can show that for any $m\leq n$ one has
 \begin{equation*}\label{eqnellipticanomalyregularizdedintegraloperatorrelation}
 \partial_{\mathbold{A}_{b}}\dashint_{E_{i_{m}}}\cdots \dashint_{E_{i_{1}}}\Phi=
\sum_{\substack{J=(j_{1},\cdots, j_{m})\\
i_{1}< j_{1},\, i_{2}<  j_{2},\,\cdots, \, i_{m}< j_{m}}}
 \partial_{\mathbold{A}_{b}}
 R^{(i_{m})}_{j_{m}}\circ \cdots\circ R^{(i_{1})}_{j_{1}}
(D_{i_{m}}^{-1}\circ \cdots\circ  D_{i_{1}}^{-1} \Phi)\,,
\end{equation*}
Here as in Remark \ref{remindependenceofantiderivatives}, $D_{i}^{-1}=\partial_{\widehat{Z}_{i}}^{-1}$
is the anti-derivative in $\widehat{Z}_{i}$.
Note that
a summand on the right hand side is nonzero if and only if $b\in  J- (I\cup\{n\})$.\\

Similar to \eqref{eqnHAEintegrated} in Theorem \ref{thmHAEBCOV},
the expansion for the iterated regularized integral $\dashint_{E_{[m]}}\Phi$ in terms of the generators
$\{\mathbold{A}_{b}, b\neq n\}$
 is also highly structured though admittedly  more complicated, reflecting the complexity of the possible degenerations
that can occur in the big diagonal $\Delta\subseteq E_{[n]}$.
The regularized integral $\dashint_{E_{[m]}}\Phi$
can then be regarded
as a sum of contributions from
various strata, where the summands are weighted by suitable
monomials in $\{\mathbold{A}_{b}, b\neq n\}$  that record the degeneration pattern
in each stratum in the big diagonal.
In physics terminology \cite{Douglas:1995conformal, Dijkgraaf:1997chiral}, this provides a coherent way of firstly
regularizing the divergences in the integrals via the 'point-splitting scheme' then picking counter terms that renormalize the integrals.
Our notion of regularized integral therefore produces a particularly nice procedure that respect almost-ellipticity and is of pure weight.

\section{Ordered $A$-cycle integrals on elliptic curves}
\label{secAcycleintegrals}

The ordered $A$-cycle integrals of
quasi-elliptic functions also admit nice residue formulas, similar to those for the regularized integrals in Proposition \ref{propresidueformulaforregularizedintegral}.
We shall derive these residue formulas and establish the connection between residue formulas for the two kinds of integrals.
Based on them, we
shall offer simpler proofs of some old results on ordered $A$-cycle integrals given in
 \cite{Goujard:2016counting, Oberdieck:2018, Li:2020regularized}.

\subsection{Residue formulas for ordered $A$-cycle integrals on elliptic curves}
\label{secresidueformulasAcycleintegrals}

In order to derive residue formulas for ordered $A$-cycle integrals,
it is convenient to introduce the following definition, following the notations in Section \ref{secpreminaries}.

\begin{dfn}\label{eqndifferenceoperatordfn}
Let $f\in\mathcal{J}_{1}$ be a quasi-elliptic function. Define the difference operator by
\[\delta: f\mapsto {1\over 2\pi i }(f(z-\tau)-f(z))\,.\]
\end{dfn}
According to Definition \ref{dfnalmostelliptic}, a quasi-elliptic function $\Psi$ is a polynomial in $Z$:
\begin{equation}\label{eqnstructureresultquasielliptic}
\Psi=\sum_{k} \Psi_{k}Z^{k}\,,\quad \Psi_{k}\in \mathcal{J}_{1}\otimes \mathbb{C}[E_{2}]\,.
\end{equation}
Define the operator $\Delta$ on the set of polynomials in $Z$ (with elliptic functions as coefficients) by linearly extending the following operator on monomials
\begin{equation}\label{eqndfnDelta}
\Delta: ({1\over 2\pi i}Z)^{k}\mapsto  ({1\over 2\pi i }Z+1)^{k}- ({1\over 2\pi i}Z)^{k}\,.
\end{equation}
Then from the functional equation
\eqref{eqnZtransform}
of $Z$,
 one has
\[\delta=\Delta\,\quad
\mathrm{when~ acting ~on}\,~\widetilde{\mathcal{J}}_{n}\,.\]

Let $\Delta^{-1}f$ be any $\Delta$-primitive of a quasi-elliptic function $f$, that is, $\Delta(\Delta^{-1}f)=f$.
We now derive a candidate for
$\Delta^{-1}{Z^{n}}$
in terms of quasi-elliptic functions using the generating series technique.
Recall that the Bernoulli polynomials $\{B_{n}(s)\}_{n\geq 0}$ are defined by
\[
\sum_{n\geq 0}B_{n}(s) {T^{n}\over n!}={T\over e^{T}-1}\cdot e^{sT}\,.
\]
In terms of the ordinary Bernoulli numbers $B_{n}=B_{n}(0)$, one has
\[
B_{n}(s)=
\sum_{k+\ell=n} {n\choose k}B_{k} {s^{\ell}}\,,\quad n\geq 0\,.
\]

\begin{lem}\label{lemexistenceofDeltainverse}
Let the notations be as above.
Then one has
\[
\Delta\left( {1\over (n+1)!}B_{n+1}({Z\over 2\pi i})\right)={Z^{n}\over (2\pi i)^{n} n!}\,,\quad n\geq 0\,.
\]

\end{lem}
\begin{proof}
Observe that
\[
\sum_{n\geq 0}\left(B_{n}(s+1)-B_{n}(s)\right) {T^{n}\over n!}={T\over e^{T}-1}\cdot (e^{(s+1)T}-e^{sT})
=T e^{sT}\,.
\]
It follows that
\[
B_{n}(s+1)-B_{n}(s)=n{s^{n-1}}\,,\quad n\geq 1\,.
\]
Setting $s=(1/2\pi i)Z$, one thus obtains
the desired claim.
\end{proof}

We have the
following result for $A$-cycle integrals that is
 analogous to the one for regularized integrals given in Lemma
\ref{lemregularizedintegralintermsofresidue}.
\begin{lem}\label{lemAcycleintegralofquasiellipticintermsofresidue}
Let $\Psi\in\widetilde{\mathcal{J}}_{1}$ be a quasi-elliptic function as given in \eqref{eqnstructureresultquasielliptic}.
Then one has
\[\int_{A}\Psi\,dz
=\sum_{k}\mathrm{Res} (\Psi_{k} \cdot \Delta^{-1}{Z}^{k}  )\,,
\]
where $\Delta^{-1}Z^{k}$ is any quasi-elliptic function $g_k\in \widetilde{\mathcal{J}}_1$ satisfying
$\Delta g_k=Z^{k}$.

\end{lem}
\begin{proof}
The existence of $g_{k}\in \widetilde{\mathcal{J}}_{1}$ follows from Lemma \ref{lemexistenceofDeltainverse}.
For any such $g_{k}\in \widetilde{\mathcal{J}}_1$, by \eqref{eqnZtransform}
one has $g_{k}(z+1)=g_{k}(z)$.
It follows that
\[-{1\over 2\pi i}\int_{A^{+}-A^{-}} \Psi_{k}\cdot g_{k}
=
-{1\over 2\pi i}\int_{A} \Psi_{k}( g_{k}(z)-g_{k}(z-\tau))
=\int_{A}\Delta(\Psi_{k}g_{k})=\int_{A}\Psi_{k}Z^{k}\,,\]
where $A^{+}$ is the cycle $A$ and $A^{-}$
the segment connecting $0$ to $1$ on the universal cover of $E_{\tau}$.
By Stokes theorem (or reciprocity law for Abelian differentials), one has
\[\mathrm{Res} (\Psi_{k}\cdot g_{k})=
-{1\over 2\pi i}\int_{A^{+}-A^{-}} \Psi_{k}\cdot g_{k}\,.\]
Summing over $k$ yields the desired claim.

\end{proof}

As explained in Remark \ref{remindependenceofantiderivatives}, different choices for the $\Delta$-primitive  yield identical results by the global residue theorem.
Similar to
\eqref{eqnregularizedintegralintermsofresidue}, the residue formula in
Lemma \ref{lemAcycleintegralofquasiellipticintermsofresidue}
can be expressed more succinctly as the following relation between operators on quasi-elliptic functions
\begin{equation}\label{eqnAcycleintegralintermsofresidue}
\int_{A}(-)
= \mathrm{Res}\, (\Delta^{-1}-)\,.
\end{equation}

Next we derive
residue formulas for ordered $A$-cycle integrals following the same strategy for the derivation of those for regularized integrals.
The results are analogues of Lemma \ref{lempreservingalmostellipticity}, Proposition \ref{propresidueformulaforregularizedintegral} and Corollary
\ref{corinductiveresidueformulaforergularizedintegral}, and are reformulations of results from \cite{Goujard:2016counting, Oberdieck:2018}.

\begin{prop}\label{propresidueformulasforAcycleintegral}
\begin{enumerate}[(i).]
\item For any $\Psi\in \widetilde{\mathcal{J}}_{[n],k}$, one has
\[
\int_{A_{i}}\Psi\,dz_{i}\in \widetilde{\mathcal{J}}_{[n]-\{i\},\leq k}\,.
\]
In particular, one has $\int_{A_{n}}\int_{A_{n-1}}\cdots \int_{A_{1}}\Psi\in \mathbb{C}(E_{4},E_{6})[{E}_{2}]$.
\item
Let $\Phi$ be an elliptic function in $\mathcal{J}_{n}$. Then one has
 \begin{equation}
 \int_{A_{[n]}}\Phi=\dashint_{A_{n}}
\sum_{\substack{\mathbold{r}=(r_{1},\cdots,\, r_{n-1})\\
r_{1}>1,\, r_{2}>2,\cdots,\, r_{n-1}=n}} R^{(n-1)}_{r_{n-1}}\circ \cdots\circ R^{(1)}_{r_{1}}
(\Phi G_{\mathbold{r}})\,.
\end{equation}
Here the function $G_{\mathbold{r}}
$ is any quasi-elliptic function satisfying
\begin{equation*}
\Delta_{n-1}\cdots \Delta_{{1}}\,G_{\mathbold{r}}=1\,,
\end{equation*}
where $\Delta_{i}$ is defined as in \eqref{eqndfnDelta} but
with  $Z$ replaced by
$Z_{i}$.
\item
Let $\Phi$ be an elliptic function in $\mathcal{J}_{n}$. Then one has
\[{1\over n!}\sum_{\sigma\in\mathfrak{S}_{n}}\int_{A_{\sigma([n])}}\Phi
=\sum_{I:\, i_{1}=1}{1\over (n-|I|)!}
\sum_{\tau}\int_{A_{\tau(I^{c})}}R^{I}_{I[1]}({Z_{i_{1}}^{m}\over m!}\Phi)\,.
\]
Here
$I=(i_{1},\cdots, i_{m})$ is a non-recurring sequence of length $m\geq 1$ that satisfies $i_{1}=1, i_{k}\neq n, 1\leq k\leq m$,
$I[1]:=(i_{2},\cdots, i_{m},i_{m+1}:=n)$,
and the summation over $\tau$ is over all possible permutations
of $I^{c}:=[n]-I$.
\end{enumerate}

\end{prop}

\begin{proof}
\begin{enumerate}[(i).]
\item
The proof of the claim
is similar to that of Lemma \ref{lempreservingalmostellipticity},
with the generator $\widehat{Z}_{a}=\widehat{Z}(z_{a}-z_{n})$ replaced
by  $Z_{a}=Z(z_{a}-z_{n})$, and Lemma
\ref{lemregularizedintegralintermsofresidue} replaced by Lemma
\ref{lemAcycleintegralofquasiellipticintermsofresidue}.
\item
The proof of Proposition
\ref{propresidueformulaforregularizedintegral}
carries sentence by sentence, with the generator $\widehat{Z}_{a}=\widehat{Z}(z_{a}-z_{n})$ replaced
by  $Z(z_{a}-z_{n})$, the
differential operator
$\partial_{\widehat{Z}}$ replaced by the difference operator $\Delta$,
and
Lemma \ref{lemregularizedintegralintermsofresidue}
replaced by Lemma \ref{lemAcycleintegralofquasiellipticintermsofresidue}.
That the last integration is trivial follows from (i) that
ordered $A$-cycle integrals preserve quasi-ellipticity and hence before the last integration the result is independent of $z$.
The independence on the choice of  the $\Delta$-primitive again follows from the global residue theorem, as mentioned earlier in Remark \ref{remindependenceofantiderivatives}.
\item
This is
 \cite[Proposition 10]{Oberdieck:2018}.
\end{enumerate}
\end{proof}

\subsection{Relation between ordered $A$-cycle integrals and regularized integrals}

Let $\Phi \in \mathcal{J}_{n,k}$, then $\int_{A_{[n]}}\Phi$ is of mixed weight with leading weight $k$.
This mixed-weight phenomenon for ordered $A$-cycle integrals was firstly discovered in \cite[Theorem 5.8]{Goujard:2016counting}, and later
proved in \cite[Theorem 7 (1)]{Oberdieck:2018} by other methods.
We have also offered a different proof using our notion of regularized integrals in our earlier work \cite[Theorem 3.9 (2)]{Li:2020regularized}.
In fact, more is true.
It is shown in \cite[Theorem 3.4 (2)]{Li:2020regularized} that the
average of ordered $A$-cycle integrals over orderings
coincides with the holomorphic limit of the corresponding regularized integral. In particular,
the average has
pure weight, as previously shown in \cite[Theorem 7 (2)]{Oberdieck:2018}. Furthermore, any ordered $A$-cycle integral admits a nice combinatorial
formula \cite[Theorem 3.9 (1)]{Li:2020regularized} in terms of holomorphic limits of various regularized integrals,  providing
explicit formulas for the components of various weights.

 We now offer more conceptual and shorter proofs of part of these results, based on
the residue formulas in  Proposition
\ref{propresidueformulaforregularizedintegral}, Corollary
\ref{corinductiveresidueformulaforergularizedintegral}  and Proposition
\ref{propresidueformulasforAcycleintegral}.

\begin{prop}\label{propweightofoperators}
The following statements hold.
\begin{enumerate}[(i).]

 \item \cite[Theorem 3.4 (1) (3), Theorem 3.9 (2)]{Li:2020regularized}
The regularized integral $\dashint_{E_{[n]}}$ of an almost-elliptic function $\widehat{\Psi}\in\widehat{\mathcal{J}}_{[n],k}$  is almost-elliptic of pure weight $k$.
While
any ordered $A$-cycle integral of a quasi-elliptic function $\widetilde{\Psi}\in\widetilde{\mathcal{J}}_{[n],k}$ is quasi-elliptic of mixed weight,
with leading weight $k$.

\item
 \cite[Theorem 3.4 (2)]{Li:2020regularized}
 Let $\Phi$ be an element in
$\mathcal{J}_{n}\otimes \mathbb{C}[\widehat{E}_{2}]$.
Then
one has
\[
\lim_{\mathbold{Y}=0}\dashint_{E_{[n]}}\Phi
=
{1\over n!}
 \sum_{\sigma\in \mathfrak{S}_{n}}\int_{A_{\sigma([n])}}\lim_{\mathbold{Y}=0}\Phi\,.
\]
Here $\lim_{\mathbold{Y}=0}$ is the holomorphic limit map on $\widehat{\mathcal{J}}_{n}\subseteq \widetilde{\mathcal{J}}_{n}\otimes \mathbb{C}[\mathbold{A}_{ij}, \mathbold{Y}]$,
obtained by taking the degree-0 term in $\mathbold{Y}$.

\end{enumerate}
\end{prop}

\begin{proof}
\begin{enumerate}[(i).]
\item
Let as before $D_{i}=\partial_{\widehat{Z}_{i}}$ and  $\Delta_{i}$ be the difference operator defined as in
\eqref{eqndfnDelta} with respect to $Z_i$.
Comparing Proposition
\ref{propresidueformulaforregularizedintegral} with
Proposition
\ref{propresidueformulasforAcycleintegral} (2), we see that the regularized integral $ \dashint_{E_{[n]}}\widehat{\Psi}$ and ordered $A$-cycle integral
 $\int_{A_{[n]}}\widetilde{\Psi}$
are sums of results obtained from the action of the same residue operators on different functions $F_{\mathbold{r}}(\widehat{\Psi}),G_{\mathbold{r}}(\widetilde{\Psi})$:
\begin{eqnarray*}
R^{[n-1]}_{\mathbold{r}}F_{\mathbold{r}}(\widehat{\Psi})&=&R^{(n-1)}_{r_{n-1}}
D_{n-1}^{-1}\cdots
R^{(2)}_{r_{2}}D_{2}^{-1} R^{(1)}_{r_{1}}D_{1}^{-1}\widehat{\Psi}\,,\\
R^{[n-1]}_{\mathbold{r}}G_{\mathbold{r}}(\widetilde{\Psi})&=&R^{(n-1)}_{r_{n-1}}
{\Delta}_{n-1}^{-1}\cdots
R^{(2)}_{r_{2}}{\Delta}_{2}^{-1} R^{(1)}_{r_{1}}{\Delta}_{1}^{-1}\widetilde{\Psi}\,.
\end{eqnarray*}
That the operator $D_{i}^{-1}$ is of pure weight $1$ follows from its definition and the polynomial structure  in $\widehat{Z}_{i}$ of
almost-elliptic functions given in Definition \ref{dfnalmostelliptic}. This tells that $R^{[n-1]}_{\mathbold{r}}F_{\mathbold{r}}(\widehat{\Psi})$
has pure weight $k$ since the residue operator has pure weight $-1$.
That
$\Delta_{i}^{-1}$ is of mixed weight and its leading weight is that of $D_{i}^{-1}$
follow from Lemma \ref{lemexistenceofDeltainverse} (cf. \cite[Proposition 5.9]{Goujard:2016counting}).

\item

We shall prove the desired claim by induction on $n$, based on Corollary
\ref{corinductiveresidueformulaforergularizedintegral} and Proposition \ref{propresidueformulasforAcycleintegral} (iii)
 proved by Oberdieck-Pixton in \cite[Proposition 10]{Oberdieck:2018}.

We start by proving a preparatory result.
Let  $I=(i_1=1,\cdots, i_{m})$ be the chain introduced in Corollary
\ref{corinductiveresidueformulaforergularizedintegral} and $R^{I}_{I[1]}=R^{(i_{m})}_{n}\circ \cdots\circ R^{(1)}_{i_{2}}$.
Note that
$F_{I}$ is a polynomial in $\widehat{Z}_{i_{k},i_{k+1}},1\leq k\leq m$ with coefficients in $\mathcal{J}_{n}$.
In the evaluation of the $R^{(i_{k})}_{i_{k+1}}$-step of $R^{I}_{I[1]} (\Phi F_{I})$, we expand the non-holomorphic quantity $\widehat{Z}_{i_{k},i_{k+1}}$ appearing in $F_{I}$ as
\[\widehat{Z}_{i_{k},i_{k+1}}=Z(z_{i_{k}}-z_{n})-Z(z_{i_{k+1}}-z_{n})+\mathbold{Y} (\overline{z}_{i_{k},i_{k+1}}-z_{i_{k},i_{k+1}})\,.\]
The contribution of terms involving $\mathbold{Y} \overline{z}_{i_{k},i_{k+1}}$
to $R^{I}_{I[1]} (\Phi F_{I})$ is zero by \eqref{eqantiholomorphicresidue}.
The same reasoning using the chain structure of $I$ tells that $R^{I}_{I[1]} (\Phi F_{I})$ is meromorphic in $z$.
In particular, by
Lemma \ref{lempreservingalmostellipticity}  we have
$R^{I}_{I[1]} (\Phi F_{I})\in \mathcal{J}_{[n]-I}\otimes \mathbb{C}[\widehat{E}_{2}]$.
Furthermore, the contribution of terms involving $\mathbold{Y} z_{i_{k},i_{k+1}}$ to $R^{I}_{I[1]} (\Phi F_{I})$ is a polynomial in $\mathbold{Y}$ of degree $\geq 1$ with coefficients lying in
$\widetilde{\mathcal{J}}_{[n]-I}$, and thus vanishes in the holomorphic limit.
It follows that
\[
\lim_{\mathbold{Y}=0} R^{I}_{I[1]} (\Phi F_{I})=R^{I}_{I[1]} (\lim_{\mathbold{Y}=0}\Phi \cdot F_{I}|_{\mathbold{A}_{ij}=0})\,.
\]
We have shown earlier in Corollary
\ref{corinductiveresidueformulaforergularizedintegral}
 that $F_{I}$ can be taken to be
$\widehat{Z}_{i_{1}}^{m}/m!$,
thus
\begin{equation}\label{eqmcommutinglimit}
\lim_{\mathbold{Y}=0} R^{I}_{I[1]} (\Phi F_{I})=R^{I}_{I[1]} (\lim_{\mathbold{Y}=0}\Phi \cdot {Z_{i_{1}}^{m}\over m!})\,.
\end{equation}

The desired claim is true for $n=2,3$ by straightforward computations.
Assuming now the statement holds for $\leq n-1$ iterated regularized integrals.
From
Corollary
\ref{corinductiveresidueformulaforergularizedintegral} one has  (note the triviality of the last integration  $\dashint_{E_{n}}$)
\[\lim_{\mathbold{Y}=0}\dashint_{E_{[n]}}\Phi
=
 \dashint_{E_{n}}\lim_{\mathbold{Y}=0}
\sum_{\substack{I:\,I\subseteq [n-1]\\i_{1}=1}}\dashint_{E_{[n-1]-I}} R^{I}_{I[1]} (\Phi F_{I})\,.
\]
By the fact $R^{I}_{I[1]} (\Phi F_{I})\in \mathcal{J}_{[n]-I}\otimes \mathbb{C}[\widehat{E}_{2}]$
and the induction hypothesis, we  obtain
\begin{eqnarray*}
&&\lim_{\mathbold{Y}=0}\dashint_{E_{[n]}}\Phi
=\dashint_{E_{n}}
 \sum_{\substack{I:\, |I|\subseteq n-1\\ i_{1}=1}}
 {1\over (n-1-|I|)!} \sum_{\tau\in \mathfrak{S}_{[n-1]-I}}\int_{A_{\tau(I^{c})}} \lim_{\mathbold{Y}=0} R^{I}_{I[1]} (\Phi F_{I})\,,
\end{eqnarray*}
where $\mathfrak{S}_{[n-1]-I}$ is the permutation group of the
index set $I^{c}:=[n-1]-I$, and $A_{\tau(I^{c})}$
is the corresponding iterated $A$-cycle ordered by the permutation $\tau$.
From \eqref{eqmcommutinglimit},
it follows that
\begin{eqnarray*}\lim_{\mathbold{Y}=0}\dashint_{E_{[n]}}\Phi
=\dashint_{E_{n}}
 \sum_{\substack{I:\, |I|\subseteq n-1\\ i_{1}=1}}
 {1\over (n-1-|I|)!} \sum_{\tau\in \mathfrak{S}_{[n-1]-I}}\int_{A_{\tau(I^{c})}}R^{I}_{I[1]} ( \lim_{\mathbold{Y}=0} \Phi \cdot {Z_{i_{1}}^{m} \over m!})\,.
\end{eqnarray*}
Applying Proposition \ref{propresidueformulasforAcycleintegral} (iii)
, the summation  on the right hand side is exactly
\[
\dashint_{E_{n}}{1\over (n-1)!}\sum_{\sigma\in \mathfrak{S}_{n-1}}\int_{A_{\sigma([n-1])}}\lim_{\mathbold{Y}=0}\Phi
=
\int_{A_{n}}{1\over (n-1)!}\sum_{\sigma\in \mathfrak{S}_{n-1}}\int_{A_{\sigma([n-1])}}\lim_{\mathbold{Y}=0}\Phi
=
{1\over n!}\sum_{\sigma\in \mathfrak{S}_{n}}\int_{A_{\sigma([n])}}\lim_{\mathbold{Y}=0}\Phi\,.\]
Here we have used the triviality of the last integrations $\int_{A_{n}}, \dashint_{E_{n}}$, and
the cyclic symmetry \cite[Remark 3.6]{Li:2020regularized} of ordered $A$-cycle integrals.
Hence the desired statement holds for the $n$ case
and the proof is  complete.

\end{enumerate}
\end{proof}

\begin{rem}\label{remequivalence}
The above reasoning in the proof of Proposition \ref{propweightofoperators} (iii) can be reversed.
Therefore, with Corollary \ref{corinductiveresidueformulaforergularizedintegral} granted,
 Proposition
\ref{propweightofoperators} (ii)  is actually equivalent to Proposition \ref{propresidueformulasforAcycleintegral} (iii).
 \end{rem}

Proposition \ref{propweightofoperators} establishes the connection between
regularized integrals and averaged ordered $A$-cycle integrals.
Based on this connection, Lemma \ref{lempreservingalmostellipticity},
and the isomorphism  \eqref{eqnhollimitmap},
one can show that the holomorphic anomaly equation \eqref{eqnHAE} in  Theorem \ref{thmHAE} for regularized integrals
is equivalent to the one
\cite[Theorem 7 (3)]{Oberdieck:2018} for averaged ordered $A$-cycle integrals.
The details are as follows. Let $\Phi\in \mathcal{J}_{n}\otimes \mathbb{C}[\widehat{E}_{2}]$.
By the polynomial structure of $\dashint_{E_{[n]}}\Phi$ in $\widehat{E}_{2}$ given in Lemma \ref{lempreservingalmostellipticity},
it  is direct to see that
\[
 \lim_{\mathbold{Y}=0}\,\partial_{\mathbold{Y}}\dashint_{E_{[n]}}\Phi=\partial_{\eta_{1}} \lim_{\mathbold{Y}=0}
 \,\dashint_{E_{[n]}}\Phi\,,\quad \eta_{1}={\pi^2\over 3}E_{2}\,.
\]
By Proposition \ref{propweightofoperators} (ii), one then has
\[
 \lim_{\mathbold{Y}=0}\,\partial_{\mathbold{Y}}\dashint_{E_{[n]}}\Phi=\partial_{\eta_{1}} \left(
 {1\over n!}
 \sum_{\sigma\in \mathfrak{S}_{n}}\int_{A_{\sigma([n])}}\lim_{\mathbold{Y}=0}\Phi\right)\,.
 \]
 Similarly,
 \[  \lim_{\mathbold{Y}=0}\,\dashint_{E_{[n]}}\partial_{\mathbold{Y}}\Phi=
  {1\over n!}
 \sum_{\sigma\in \mathfrak{S}_{n}}\int_{A_{\sigma([n])}}\lim_{\mathbold{Y}=0}\partial_{\mathbold{Y}} \Phi
 =
   {1\over n!}
 \sum_{\sigma\in \mathfrak{S}_{n}}\int_{A_{\sigma([n])}}\partial_{\eta_1} \lim_{\mathbold{Y}=0} \Phi
 \,.
\]
\begin{eqnarray*}
&&
 \lim_{\mathbold{Y}=0}\left(-\sum_{a,b:\,a< b} \dashint_{E_{[n]-\{a\}}} R^{(a)}_{b}( (z_{a}-z_{b})\Phi)\right)\nonumber\\
 &=&
-\sum_{a,b:\,a< b} {1\over (n-1)!}  \sum_{\sigma\in \mathfrak{S}_{n-1}}\int_{A_{\sigma([n]-\{a\})}} \lim_{\mathbold{Y}=0} R^{(a)}_{b}( (z_{a}-z_{b})\Phi)\,\nonumber\\
 &=&
-\sum_{a,b:\,a< b} {1\over (n-1)!}  \sum_{\sigma\in \mathfrak{S}_{n-1}}\int_{A_{\sigma([n]-\{a\})}} R^{(a)}_{b}( (z_{a}-z_{b})\lim_{\mathbold{Y}=0} \Phi)\,.
\end{eqnarray*}
In the last equality we have used the fact that $\Phi$ is meromorphic in $z$.
Taking the holomorphic limit of  \eqref{eqnHAE}, one then has
\begin{eqnarray*}
\partial_{\eta_{1}} \left(
 {1\over n!} \sum_{\sigma\in \mathfrak{S}_{n}}\int_{A_{\sigma([n])}}\lim_{\mathbold{Y}=0}\Phi\right)
 &=&   {1\over n!}
 \sum_{\sigma\in \mathfrak{S}_{n}}\int_{A_{\sigma([n])}}\partial_{\eta_1} \lim_{\mathbold{Y}=0} \Phi\nonumber\\
& -&\sum_{a,b:\,a< b} {1\over (n-1)!}  \sum_{\sigma\in \mathfrak{S}_{n-1}}\int_{A_{\sigma([n]-\{a\})}} R^{(a)}_{b}( (z_{a}-z_{b})\lim_{\mathbold{Y}=0} \Phi)\,.
\end{eqnarray*}
This is the anomaly equation \cite[Theorem 7 (3)]{Oberdieck:2018} for the averaged ordered $A$-cycle integrals of the quasi-elliptic function $\lim_{\mathbold{Y}=0} \Phi\in \mathcal{J}_{n}\otimes\mathbb{C}[E_{2}]$.
The converse implication follows from the fact that the holomorphic limit map
$ \lim_{\mathbold{Y}=0}$ in \eqref{eqnhollimitmap} induces isomorphisms $\mathbb{C}(E_{4},E_{6})[\widehat{E}_{2}]\cong
\mathbb{C}(E_{4},E_{6})[{E}_{2}],
\mathcal{J}_{n}\otimes\mathbb{C}[\widehat{E}_{2}]\cong \mathcal{J}_{n}\otimes\mathbb{C}[E_{2}]$.\\


The relations \cite[Theorem 3.4, Theorem 3.9]{Li:2020regularized} between regularized integrals and ordered $A$-cycle integrals (with or without averaging)
follow from the Poincar\'e-Birkhoff-Witt Theorem relating
tensor algebra and universal enveloping algebra of free Lie algebra.
It seems interesting to derive these relations based on
the combinatorial relations between the operators $R^{[n-1]}_{\mathbold{r}}F_{\mathbold{r}}(-)$ and
$R^{[n-1]}_{\mathbold{r}}G_{\mathbold{r}}(-)$.
See the discussions in Appendix \ref{appendixcombinatorial} for instance.
Another interesting direction is to
study the connection to the Baker--Campbell--Hausdorff--Dynkin formula and Chen's iterated integrals.

\begin{appendices}

\section{Some combinatorial properties of residue formulas}
\label{appendixcombinatorial}

We discuss some combinatorial properties of residue formulas and their applications in this Appendix.

\subsection{Further commutation relations of residue operators}

We first give an extension of Lemma \ref{lemcommutatorresidue}.

\begin{lem}\label{lemcommutatorresiduec=0}
Introduce the notation $R^{(a)}_{0}=\int_{A} dz_{a}$.
Then one has the following identities of operators acting on quasi-elliptic functions in $z$
\begin{enumerate}[(i).]
\item
\begin{equation}\label{eqncommutatorofR0}
R^{(a)}_{c}R^{(b)}_{c}-R^{(b)}_{c}R^{(a)}_{c}=R^{(b)}_{c}R^{(a)}_{b}
=-R^{(a)}_{c}R^{(b)}_{a}\,,\quad
a,b\in [n]\,,~ c\in[n] \cup\{0\}\,.
\end{equation}
\item
\begin{equation}\label{eqntrivialcommutatorofR0}
R^{(e)}_{c}R^{(a)}_{b}=R^{(a)}_{b}R^{(e)}_{c}\,, \quad \mathrm{if}\,
\quad a,b,e\in [n]\,, c\in [n]\cup \{0\}\,,\quad  \{a,b\}\cap \{c,e\}=\emptyset\,.
\end{equation}
\item When acting on almost-elliptic functions in $z$, one has
\begin{equation}\label{eqnArnold}
R^{(b)}_{c}R^{(a)}_{b}+R^{(c)}_{a}R^{(b)}_{c}+R^{(a)}_{b}R^{(c)}_{a}=0\,, \quad \mathrm{if}\,
\quad a,b,c\in [n]\quad \text{are distinct}\,.
\end{equation}
\end{enumerate}
Here the various operators are understood to be composed in the natural operator ordering.

\end{lem}
\begin{proof}
The first two identities with $c\neq 0$ are covered in Lemma \ref{lemcommutatorresidue}.
The first identity \eqref{eqncommutatorofR0} with $c=0$ is proved by using the standard integral contour deformation argument on the complex plane.
For the identity \eqref{eqntrivialcommutatorofR0},
by the  global residue theorem
 and Lemma \ref{lemcommutatorresidue} one has
\begin{eqnarray*}
R^{(e)}_{0}R^{(a)}_{b}
&=&-(\sum_{\substack{i:\, e\prec i\\ i\neq b}} R^{(e)}_{i}
+R^{(e)}_{b})
R^{(a)}_{b}\nonumber\\
&=&-R^{(a)}_{b} \sum_{\substack{i:\, e\prec i\\ i\neq b }} R^{(e)}_{i}
-R^{(a)}_{b}R^{(e)}_{b}-R^{(a)}_{b}R^{(e)}_{a}+R^{(a)}_{b}R^{(e)}_{b}+R^{(a)}_{b}R^{(e)}_{a}-R^{(e)}_{b}
R^{(a)}_{b}\nonumber\\
&=&R^{(a)}_{b}  R^{(e)}_{0}\,.
\end{eqnarray*}
Here in the first equality we have used the fact that $R^{(a)}_{b}$ preserves meromorphicity,
so that the global residue theorem can be applied to $R^{(e)}_{0}R^{(a)}_{b}$.
The identity \eqref{eqnArnold} is proved by a local calculation in a similar way as in Lemma \ref{lemcommutatorresidue}.
\end{proof}

The special case  with $c=0$ in
\eqref{eqntrivialcommutatorofR0} reads that
\[
R^{(a)}_{b}\int_{A_{i}}=\int_{A_{i}}R^{(a)}_{b} \,\quad
\text{for distinct}~a,b,i\,,~\text{when acting on quasi-elliptic functions}\,.
\]
We also record the following interesting result that is parallel to the above commutation relation.
\begin{cor}\label{eqncorresiduecommuteswithregularizedintegral}
One has the following identity of operators
\[
R^{(a)}_{b}\dashint_{E_{i}}=\dashint_{E_{i}}R^{(a)}_{b}\,\quad
\text{for~distinct}~a,b,i\,,~\text{when acting on almost-elliptic functions}\,.
\]
\end{cor}
\begin{proof}
One proof is already implicitly given when proving \eqref{eqnRabwabintsimplified} in Theorem \ref{thmHAE} using the residue formula for regularized integrals.
An alternative proof using the relation between regularized integrals and ordered $A$-cycle integrals is given as follows.\footnote{Lemma \ref{lempartialregularizedintegral}
can also be proved based on  \eqref{eqnalternativeexpressionvolumeform}  by using the residue calculations detailed in
Lemma \ref{lemhollimitintermsofforest} below. Together with the  reasoning here that proves  \eqref{eqnRabwabintsimplified},
the derivation of  Theorem \ref{thmHAE} actually does not rely on the explicit residue formulas (e.g., Lemma \ref{lemregularizedintegralintermsofresidue})
for regularized integrals, but only on the relation \eqref{eqnalternativeexpressionvolumeform} between regularized integrals and
ordered $A$-cycle integrals.}
For any almost-elliptic function $\Psi$,
we write
\[
\Psi=\sum_{k} \Psi_{k} ({\overline{z}_i-{z}_i\over \overline{\tau}-\tau})^k\,,
\]
where $\Psi_{k}$ is a quasi-elliptic function in $z_{i}$.
Then
using the expression of $\mathrm{vol}$ given in
Remark \ref{remvolumeformforquasiness} and the Stokes theorem (see \cite[Lemma 3.26]{Li:2020regularized}), we have
\begin{eqnarray}\label{eqnalternativeexpressionvolumeform}
\dashint_{E_{i}}\Psi=\sum_{k} {1\over k+1}\int_{A} \Psi_{k}
+2\pi i \sum_{k}{1\over k+1}\sum_{r} R^{(i)}_{r}  \left(({\overline{z}_i-{z}_i\over \overline{\tau}-\tau})^{k+1} \Psi_{k}\right)
\end{eqnarray}
By \eqref{eqnalternativeexpressionvolumeform}, linearity and Lemma \ref{lemcommutatorresiduec=0} with $c=0$, it suffices to prove the following identity
\[
R^{(a)}_{b}\sum_{r:\,r\neq i}R^{(i)}_{r} \left( (\overline{z}_{i}-z_{i})^{k+1} \Psi_k\right)
=\sum_{r:\,r\neq i,a}R^{(i)}_{r} \left( (\overline{z}_{i}-z_{i})^{k+1} R^{(a)}_{b} \Psi_k\right)\,,
\]
where $\Psi_k$ is quasi-elliptic in $z_{i}$.
This follows from Lemma \ref{lemcommutatorresidue}
that is derived by local computations and is valid
when acting on $(\overline{z}_{i}-z_{i})^{k+1} \Psi_k$.
\end{proof}

\subsection{Summation over planar rooted labeled forests}
\label{seclabeledrootedforests}

The sequence $r$
in Proposition \ref{propresidueformulaforregularizedintegral}, Proposition \ref{propresidueformulasforAcycleintegral}
 corresponds to rooted, labeled forest, as has been discussed in \cite{Li:2020regularized}.
 We now elaborate on this.
For each sequence $r$ we associate  a map
\begin{equation}\label{eqnsequencerforstandardordering}
\mathbold{r}:[n-1]\rightarrow [n]\,,\quad\quad
\mathbold{r}(j)=r_{j}\,, ~j\in [n-1]\,.
\end{equation}
satisfying
\begin{equation}\label{eqnconditinforsequencerforstandardordering}
j<\mathbold{r}(j)\,,\quad
j\in [n-1]\,.
\end{equation}
The data \eqref{eqnsequencerforstandardordering} gives rise to a rooted, labeled forest,
with the vertices labeled by $\mathbold{r}^{-1}(n)$ designated as the roots.
Define an equivalence relation $\sim $ on $[n-1]$ by declaring
\[
j\sim \mathbold{r}(j)\,,\quad j\in [n-1]\,.
\]
Then $[n-1]$ is decomposed into several disjoint subsets
\begin{equation}\label{eqndecompositionintotrees}
[n-1]=\bigsqcup_{k\in \mathbold{r}^{-1}(n)} \Gamma_{k}\,.
\end{equation}
We call $\mathbold{r}|_{\Gamma_{k}}$ a tree in the forest $\mathbold{r}$.
\\

Consider the blackboard orientation in which vertices that have valency zero or one
 are ordered from right to left, vertices  within a path in a tree are ordered from bottom to the root in the top.
Due to the constraint \eqref{eqnconditinforsequencerforstandardordering}, each sequence $\mathbold{r}$ above gives a  planar/embdedded,  rooted, labelled forest (called forest for short), whose labelings are
compatible with the blackboard orientation in the sense that with respect to $<$
\begin{itemize}
\item labelings for the roots $\mathbold{r}^{-1}(n)$ increases from right to left;
\item within a rooted tree, labelings increases from from bottom to top (towards the root);
\item within a rooted tree, labelings of children of the same vertex increase from right to left.
\end{itemize}

By Lemma \ref{lemcommutatorresidue}, when acting on almost-elliptic functions
the operator $R^{[n-1]}_{\mathbold{r}}=
R^{(n-1)}_{\mathbold{r}(n-1)}
\cdots R^{(2)}_{\mathbold{r}(2)}
R^{(1)}_{\mathbold{r}(1)}$ can be factored into a product of residue operators associated to
the trees
\begin{equation}\label{eqnresidueoperatordecomposition}
R^{[n-1]}_{\mathbold{r}}=\prod_{k}R_{\Gamma_{k}}\,.
\end{equation}
Similarly, the integration operator \eqref{eqnFrformula} factor into a product of commuting integration operators.
The condition \eqref{eqnconditinforsequencerforstandardordering} or equivalently the blackboard orientation records the operator orderings in the residue  and integration operators within each $\Gamma_{k}$ (permutations of different trees and of the leaves within the trees are not allowed by the embedding data).
In particular, when acting on an elliptic function $\Phi$ one has
 \begin{equation}\label{eqnregularizedintegralassumoverrootedtreesunderblackboardorientation}
\dashint_{E_{n}}\dashint_{E_{[n-1]}}\Phi=\sum_{\mathbold{r}}R^{[n-1]}_{\mathbold{r}}(\Phi F_{\mathbold{r}})
=\sum_{\Gamma=(\Gamma_{k})}^{<} \left(\prod_{k} R_{\Gamma_{k}}(\Phi \prod_{k}F_{\Gamma_k})\right)
\,,
\end{equation}
where the summation $\sum^{<}$ on the right hand side is over forests whose labelings are compatible with the blackboard orientation, with respect to $<$.

\begin{rem}
A rooted  labeled  forest $\mathbold{r}$ above
gives rise to a rooted labeled  tree obtained by joining all the roots in the rooted forest to the vertex $n$ now designated as the new root.
The joint tree
corresponds to an element in the basis of the cohomology group $H^{n-1}(\mathrm{Conf}_{n}(\mathbb{C}),\mathbb{C})$
whose rank is the signed first Stirling number $(-1)^{n-1}s(n,1)=(n-1)!$.
The relation
$R^{(j)}_{n}R^{(i)}_{j}=-R^{(i)}_{n}R^{(j)}_{i}$ and
 the Jacobi identity
\begin{eqnarray*}
[[R^{(a)}_{n},R^{(b)}_{n} ],R^{(c)}_{n}]+[[R^{(b)}_{n},R^{(c)}_{n} ],R^{(a)}_{n}]+[[R^{(c)}_{n},R^{(a)}_{n} ],R^{(b)}_{n}]=0\,.
\end{eqnarray*}
from Lemma  \ref{eqncommutatorofR0}  then correspond to the anti-symmetry and Arnold relation (assembling the same form as Lemma  \ref{eqncommutatorofR0} (iii)) in the cohomology ring $H^{*}(\mathrm{Conf}_{n}(\mathbb{C}),\mathbb{C})$.
See \cite{Totaro:1996, Getzler:1999} for details.
It seems that our decomposition here is related to the mixed Hodge structures on $H^{*}(\mathrm{Conf}_{n}(E),\mathbb{C})$.
We hope to discuss the algebraic formulation of regularized integrals using the language of mixed Hodge structures in a future investigation, following the lines in \cite{Totaro:1996, Getzler:1999}.
\end{rem}

There is in fact a bijection between the set of planar, rooted, labeled forests that are compatible with blackboard orientation
and $\mathfrak{S}_{[n-1]}$ given as follows.
\begin{itemize}

\item
For each $\phi\in \mathfrak{S}_{[n-1]}$, consider its one-line notation $\phi=\phi_1 \phi_2\cdots \phi_{n-1}$. Read from left to right, the maxima correspond to roots of the trees.
From the remaining entries of $\phi$,
$\phi_{a}$ is a child of $\phi_{b}$ if $\phi_{b}$ is the rightmost entry that precedes $\phi_{a}$ and is larger than $\phi_{a}$.
This gives a rooted, labeled forest and thus a unique planar one that is compatible with backboard orientation.
\item
For each planar, rooted, labeled forest compatible with backboard orientation, one
starts at the vertex $n$ of the forest and then traverses the forest
clockwise. Reading off each label as it is reached, this yields an element in $\mathfrak{S}_{[n-1]}$ in its one-line notation.

\end{itemize}
Using the above correspondence,  \eqref{eqnregularizedintegralassumoverrootedtreesunderblackboardorientation} can be alternatively written as
 \begin{equation}\label{eqnregularizedintegralassumoverSn-1}
\dashint_{E_{n}}\dashint_{E_{[n-1]}}\Phi=\sum_{\phi\in \mathfrak{S}_{[n-1]}}R^{[n-1]}_{\mathbold{r}(\phi)}(\Phi F_{\mathbold{r}(\phi)})
=\sum_{\phi\in \mathfrak{S}_{[n-1]}} \left(R_{\Gamma(\phi)}(\Phi F_{\Gamma(\phi)})\right)
\,,
\end{equation}
where $\mathbold{r}(\phi)$ and equivalently $\Gamma(\phi)$ is the forest compatible with blackboard orientation that is determined by $\phi$.

A different ordering $\prec$ is determined by $\sigma\in \mathfrak{S}_{[n-1]}$
via the integration region $A_{\sigma([n-1])}$ or $E_{\sigma([n-1])}$.
The ordering $\prec$ is related to the
magnitude ordering $<:1<2<\cdots < n-1$ by the action by $\sigma$
\begin{equation}\label{eqnlinearlorderingcorrespondingtopermutation}
\sigma_{1}\prec \sigma_{2}\prec\cdots \prec \sigma_{n-1}\,.
\end{equation}
Correspondingly, the map $\mathbold{r}$ is changed to $\mathbold{r}':\sigma([n-1])=[n-1]\rightarrow [n]$ with the convention $\sigma(n)=n$. It satisfies
\begin{equation}\label{eqnconditinforsequencerforsigmaordering}
\sigma_{i}\prec \mathbold{r}'(\sigma_{i})
\end{equation}
instead of \eqref{eqnconditinforsequencerforstandardordering}.
The set of such functions $\mathbold{r}'$ is in one-to-one correspondence with the set of the functions
$\sigma(\mathbold{r}):=\sigma\circ \mathbold{r}\circ \sigma^{-1}$, with the convention $\sigma(n)=n$.
Correspondingly, we have
\begin{equation}\label{eqnsigmaregularizedintegralassumoverr}
\dashint_{E_{n}}\dashint_{E_{\sigma([n-1])}}\Phi=
\sum_{\mathbold{r}'}
R^{\sigma([n-1])}_{\mathbold{r}'\circ \sigma}(\Phi\cdot F^{\sigma([n-1])}_{\mathbold{r}'\circ \sigma([n-1])})
=\sum_{\mathbold{r}}
R^{\sigma([n-1])}_{\sigma\circ\mathbold{r}}(\Phi\cdot F^{\sigma([n-1])}_{\sigma\circ\mathbold{r}})\,,
\end{equation}
with $\mathbold{r}$ satisfying \eqref{eqnsequencerforstandardordering}, \eqref{eqnconditinforsequencerforstandardordering}.

Similar to \eqref{eqnregularizedintegralassumoverrootedtreesunderblackboardorientation}, \eqref{eqnsigmaregularizedintegralassumoverr} is a sum over the collection of embedded, rooted, labeled forests in \eqref{eqnregularizedintegralassumoverrootedtreesunderblackboardorientation},
with the new labelings obtained by applying $\sigma$ on the original ones.
Now the new labelings are compatible with the blackboard orientation not with respect to $<$ but with respect to $\prec$ in \eqref{eqnlinearlorderingcorrespondingtopermutation}.
See Figure \ref{figure:regularizedintegralassumovergraph} below for an illustration.
	\begin{figure}[h]
		\centering
		\includegraphics[scale=0.6]{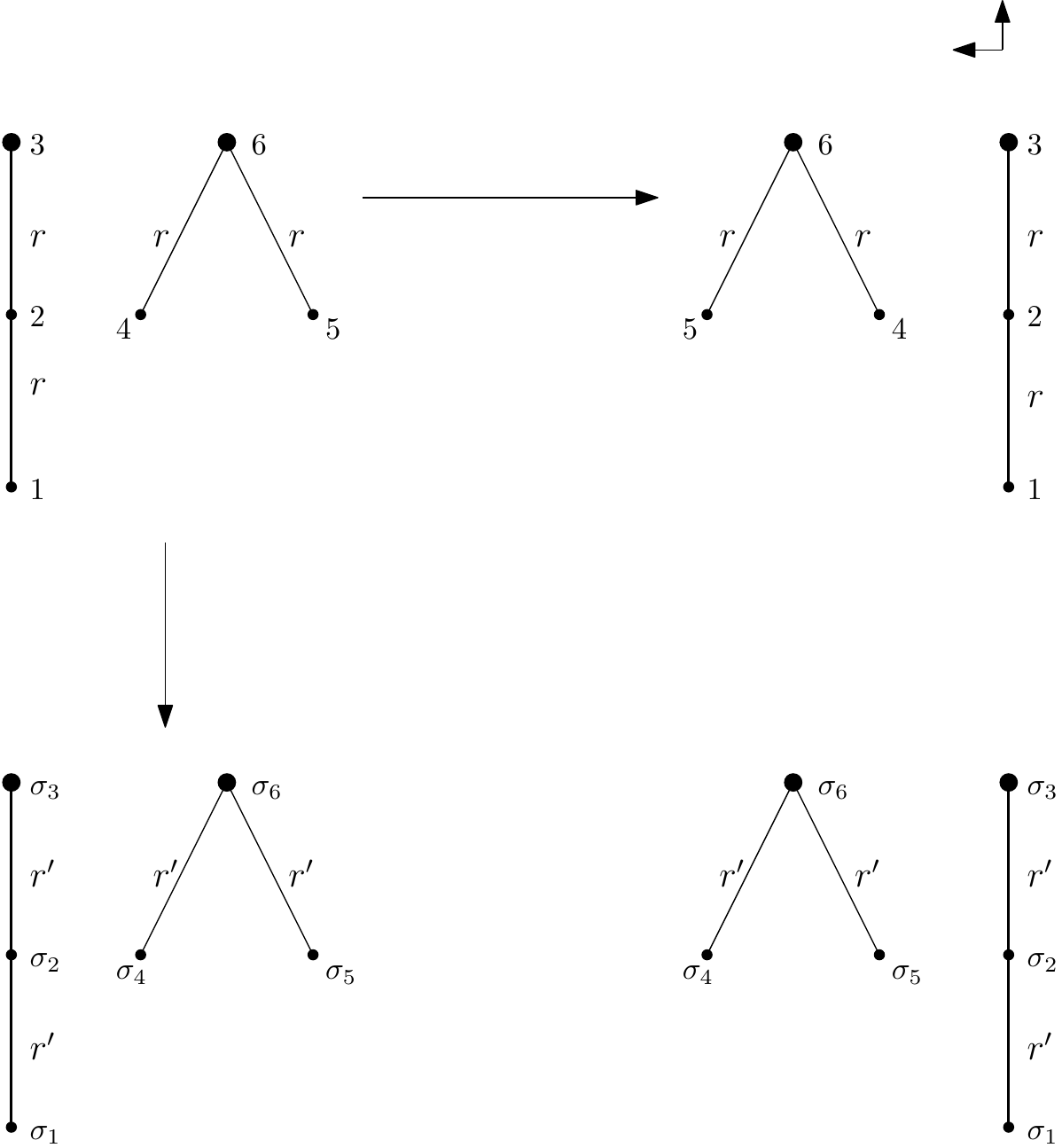}
		\caption{From left to right: picking representatives compatible with blackboard orientation. From top to bottom: applying
permutation $\sigma$ to obtain labeled rooted forests in the residue formulas of regularized integrals, with the map $\mathbold{r}'$ given by $\sigma\circ\mathbold{r}\circ \sigma^{-1}$. }
		\label{figure:regularizedintegralassumovergraph}
	\end{figure}

By the independence on the ordering of the regularized integral, we have
 \begin{equation}\label{eqnindependenceofregularizedintegralonordering}
\dashint_{E_{[n]}} \Phi={1\over (n-1)!}\dashint_{E_{n}}\sum_{\sigma\in \mathfrak{S}_{[n-1]}}\dashint_{E_{\sigma([n-1])}}\Phi\,.
\end{equation}
Now summing over the orderings $\prec$ in
\eqref{eqnlinearlorderingcorrespondingtopermutation}  is equivalent to summing over the
relabelings of the embedded, labeled rooted forests in \eqref{eqnregularizedintegralassumoverrootedtreesunderblackboardorientation}, with the operator orderings for the residue and integration operators given by the blackboard orientation.
It follows that
 \begin{equation}\label{eqnsumoverrelabelingsolabeledrootedforests}
\dashint_{E_{[n]}} \Phi
={1\over (n-1)!}
\dashint_{E_{n}}
\sum_{\Gamma=(\Gamma_{k})}^{<}
\sum_{\sigma\in\mathfrak{S}_{[n-1]}} \left(\prod_{k} R_{\sigma\circ \Gamma_{k}}(\Phi \prod_{k}F_{\sigma\circ \Gamma_k})\right)\,,
\end{equation}
where $\sigma\circ \Gamma_{k}$ is the tree whose labelings are obtained from the $\sigma$-action on
those in $\Gamma_{k}$.
Denote the set of labelings of $\Gamma_{k}$ by $\rho_{k}\subseteq [n-1]$, then the summation
can be further expressed as a sum
\[
\sum_{\Gamma=(\Gamma_{k})}^{<}
\sum_{\sigma\in\mathfrak{S}_{[n-1]}} =
\sum_{[\Gamma]=([\Gamma_{k}])}
\sum_{\rho=(\rho_{k})\in \Gamma_{[n-1]}}\sum_{\tau\in \prod_{k}\mathfrak{S}_{\rho_{k}}}\,,
\]
where $[\Gamma]$ is the forest without the labelings and $\Gamma_{[n-1]}$
is the set of compositions of $[n-1]$.
Therefore, we obtain
\begin{equation}\label{eqnsumovergammarhosigma}
\dashint_{E_{[n]}} \Phi
={1\over (n-1)!}
\dashint_{E_{n}}
\sum_{[\Gamma]=([\Gamma_{k}])}
\sum_{\rho=(\rho_{k})\in \Gamma_{[n-1]}}\sum_{\tau\in \prod_{k}\mathfrak{S}_{\rho_{k}}} R_{([\Gamma],\rho,\tau)}(\Phi \prod_{k}F_{([\Gamma],\rho,\tau)})\,.
\end{equation}
One can also exchange the order of the summation by summing over $\rho$ in the outmost one.
Then the summation is related to summation over different types of forests with fixed number of vertices in each of the trees.

\begin{ex} Consider the $n=3$ case.
The two forests in \eqref{eqnregularizedintegralassumoverrootedtreesunderblackboardorientation}
are given by the top left two in Figure \ref{figure:regularizedintegralassumovergraphn3}  below.
	\begin{figure}[h]
		\centering
		\includegraphics[scale=0.6]{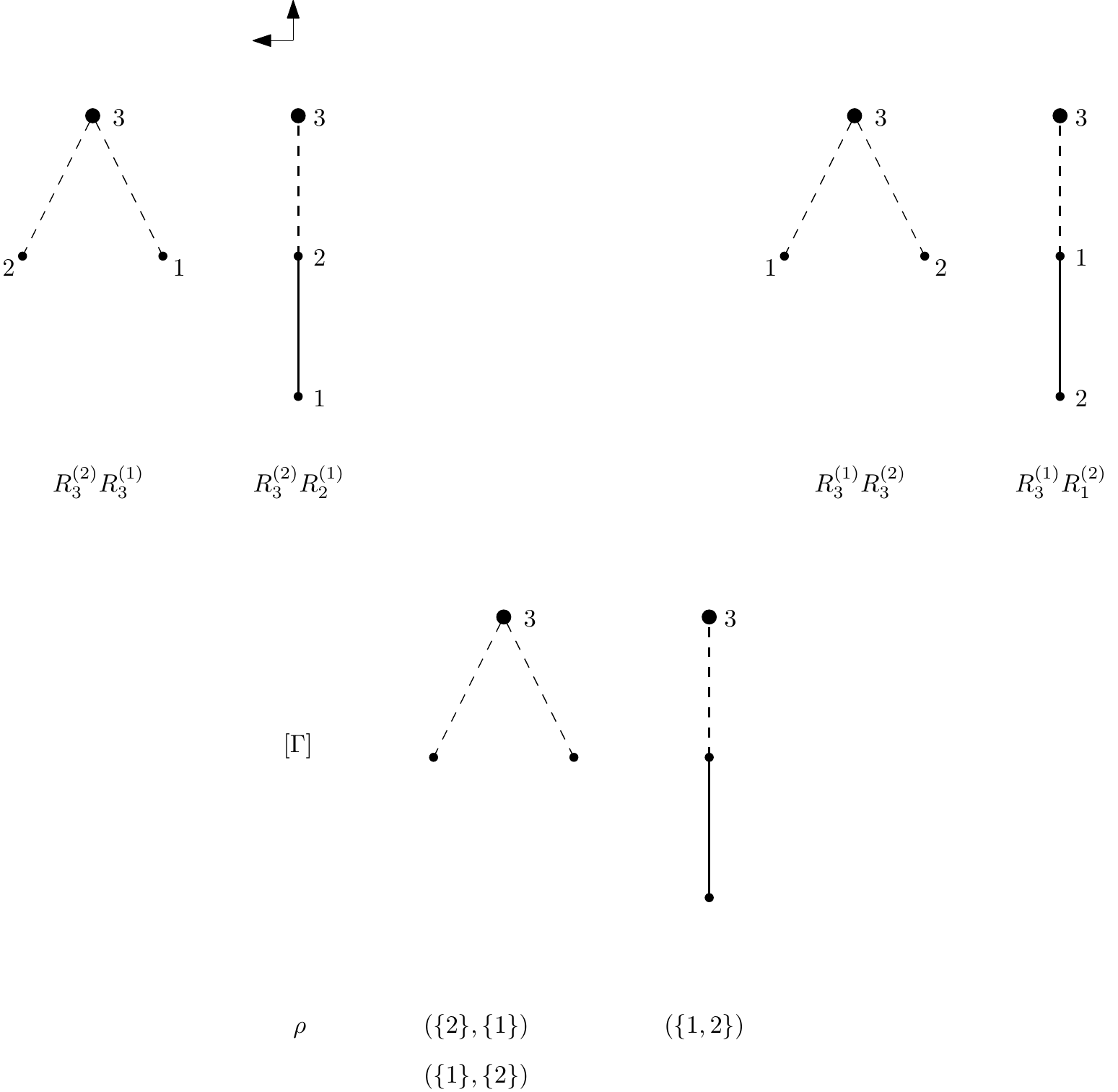}
		\caption{The left two forests correspond to summands in $\int_{E_3}\int_{E_2}\int_{E_1}$,
the right two correspond to the summands in $\int_{E_3}\int_{E_1}\int_{E_2}$. }
		\label{figure:regularizedintegralassumovergraphn3}
	\end{figure}
In this case, $\mathfrak{S}_{n-1}$ consists of two elements $\sigma=(1)(2), (12)$.
Hence the forests in \eqref{eqnsumoverrelabelingsolabeledrootedforests}
are obtained by including two more forests indicated in the top right two in
 Figure \ref{figure:regularizedintegralassumovergraphn3}.
The resummation in \eqref{eqnsumovergammarhosigma} is illustrated in the data in the bottom.
\end{ex}

Similar results hold for ordered $A$-cycle integrals, except now one does not have
\eqref{eqnindependenceofregularizedintegralonordering} for a general elliptic function $\Phi$.

The structures in \eqref{eqnregularizedintegralassumoverrootedtreesunderblackboardorientation} and \eqref{eqnsumovergammarhosigma} elucidate many
properties of regularized integrals and ordered $A$-cycle integrals.
For example, Corollary \ref{corinductiveresidueformulaforergularizedintegral} follows directly from
this and the choice $F_{I}=\widehat{Z}_{i_{1}}^{m}/m!$ given in \eqref{eqnFIsimpleintegralformula} for a non-recurring sequence $I=(i_{1},\cdots, i_{m}),i_{m+1}=n$.

\subsection{Holomorphic limit of regularized integrals}
\label{secholomorphimilitofregularized}

In  Proposition \ref{propweightofoperators} (ii), firstly established by \cite[Theorem 3.4 (2)]{Li:2020regularized},
we
used  Proposition \ref{propresidueformulasforAcycleintegral} (iii) proved by
Oberdieck-Pixton in \cite[Proposition 10]{Oberdieck:2018}.
We now give a   proof of Proposition \ref{propweightofoperators} (ii) that is simpler than  the one in  \cite{Li:2020regularized}, following the residue operator approach.
According to Remark \ref{remequivalence}, this also provides a different proof of Proposition \ref{propresidueformulasforAcycleintegral} (iii).

\begin{lem}\cite[Lemma 3.26, Lemma 3.37]{Li:2020regularized}\label{lemhollimitintermsofforest}
As operators acting on elliptic functions in
${\mathcal{J}}_{n}$, one has
\begin{equation}\label{eqnholomorphiclimitregularizedintegral}
\lim_{\mathbold{Y}=0}
\dashint_{E_{[n]}}
=\int_{A_{[n]}}\circ \lim_{\mathbold{Y}=0}+\sum_{ J:
\,J\neq\emptyset} (2\pi i)^{|J|}  \sum_{\mathfrak{r}_{J}} C_{J,\mathfrak{r}_{J}} \cdot \left(\int_{A_{[n]-J}}
R^{J}_{\mathfrak{r}_{J}}\right)^{\prec}\circ\lim_{\mathbold{Y}=0}~\,,
\end{equation}
where
\begin{itemize}
\item
$\prec$ is the prescribed ordering $<$ by magnitude as
in $\int_{A_{[n]}}$.
\item
$J=(j_1,\cdots, j_\ell)$ is an non-recurring sequence: $j_1\prec j_2\prec\cdots \prec j_\ell$ with cardinality $|J|=\ell$.
\item
$\mathfrak{r}_{J}=(\mathfrak{r}_{j_1},\cdots, \mathfrak{r}_{j_\ell})$ satisfies
$j_a\prec \mathfrak{r}_{j_{a}}\,, \mathfrak{r}_{j_{a}}\in [n]$ for
$a=1,2,\cdots,\ell$.
\item
$
C_{J,\mathfrak{r}_{J}}=\int_{0}^{1}dx_{\mathfrak{r}_{j_{\ell}}}\int_{0}^{x_{\mathfrak{r}_{j_{\ell}}}}dx_{j_{\ell}}\cdots \int_{0}^{
x_{\mathfrak{r}_{j_{1}}}}dx_{j_{1}}
$.
\item
$
R^{J}_{\mathfrak{r}_{J}}=R^{(j_\ell)}_{\mathfrak{r}_{j_\ell}}\circ\cdots \circ R^{(j_1)}_{\mathfrak{r}_{j_1}}$
and the notation
$ \left(\int_{A_{[n]-J}}
R^{J}_{\mathfrak{r}_{J}}\right)^{\prec}$ stands for the operator arranged according the ordering $\prec$.
\end{itemize}
\end{lem}
\begin{proof}
We only sketch the proof.
For any almost-elliptic function $\Psi$,
$
\Psi=\sum_{k} \Psi_{k} ({\overline{z}_i-{z}_i\over \overline{\tau}-\tau})^k\,,
$
continuing with \eqref{eqnalternativeexpressionvolumeform}, one has
\begin{eqnarray*}
&&\dashint_{E_{i}}\Psi\\
&=&\sum_{k} {1\over k+1}\int_{A} \Psi_{k}
+2\pi i \sum_{k}{1\over k+1}\sum_{r} R^{(i)}_{r}  \left(({\overline{z}_i-{z}_i\over \overline{\tau}-\tau})^{k+1} \Psi_{k}\right)\\
&=& \sum_{k} {1\over k+1}\int_{A} \Psi_{k}
+2\pi i \sum_{k}{1\over k+1}\sum_{r} R^{(i)}_{r}  \left(({\overline{z}_i-\overline{z}_r+\overline{z}_r-z_{r}+z_{r}-z_{i}\over \overline{\tau}-\tau})^{k+1} \Psi_{k}\right)\nonumber\\
&=&
\sum_{k} {1\over k+1}\int_{A} \Psi_{k}
+2\pi i \sum_{k}{1\over k+1}\sum_{\ell:\, 0\leq \ell\leq k+1}{k+1\choose \ell} ({\overline{z}_r-z_{r}\over \overline{\tau}-\tau})^{k+1-\ell}({1\over \overline{\tau}-\tau})^{\ell}  \sum_{r}R^{(i)}_{r}  \left(
(z_{r}-z_{i})^{\ell}
\Psi_{k}\right)\nonumber\,.
\end{eqnarray*}
Here in the last equality we have used $\eqref{eqantiholomorphicresidue}$.
It follows that as operators on the almost-elliptic function $\Psi$ one has
\[
\lim_{\mathbold{Y}=0}
\dashint_{E} \sum_{k} \Psi_{k} ({\overline{z}_i-{z}_i\over \overline{\tau}-\tau})^k= \int_{A}\sum_{k} {1\over k+1}\Psi_{k}+2\pi i \lim_{\mathbold{Y}=0} \sum_{r} \sum_{k} ({\overline{z}_{r}-z_{r} \over \overline{\tau}-\tau})^{k+1}\,R^{(i)}_{r} (\Psi_{k})\,.
\]
One then expands iterated regularized integrals of
almost-elliptic functions
in terms of iterated $A$-cycle integrals and residues
(see \cite[Lemma 3.26]{Li:2020regularized}).
Applying the above reasoning  to get rid of terms that do not contribute in the holomorphic limit,
the desired claim then follows,
similar to the proof of Proposition \ref{propresidueformulaforregularizedintegral}.
Note that there is no term with $J=[n]$ by the above constraint on $J,\mathfrak{r}_{J}$, as it should be the case
according to the behavior of the pole structure under residue operations.
\end{proof}

\begin{rem}
When performing the operations in Lemma \ref{lemhollimitintermsofforest}, we do not use $z_{n}$ as the reference as in Proposition \ref{propresidueformulaforregularizedintegral}, so that
 the $\mathfrak{S}_{n}$-symmetry among all the indices in $[n]$ is retained.
Of course the same relation holds if we do apply the referencing.
\end{rem}

Before proceeding to give another proof of Proposition \ref{propweightofoperators} (ii), we introduce some notations.
Similar to the discussions in Appendix \ref{seclabeledrootedforests},
for any pair $(J,\mathfrak{r}_{J})$ in  Lemma \ref{lemhollimitintermsofforest}, we
define a map
\begin{equation}\label{eqndfnforest}
\mathbold{r}:[n]\rightarrow [n]\cup\{0\}\,,\quad\quad
\mathbold{r}(j)=\mathfrak{r}_{j}\,, ~j\in J\subseteq [n]\,,\quad
\mathbold{r}(k)=0\,,~ k\in [n]-J\,.
\end{equation}
It is clear that the data $(J,\mathfrak{r}_{J})$ is equivalent to a map $\mathbold{r}$ in \eqref{eqndfnforest}
satisfying $j\prec \mathbold{r}(j)$ for any $j\in \mathbold{r}^{-1}([n])$.
In what follows we shall not distinguish them.
Denote
\begin{equation}\label{eqndfnresidueoperator}
\mathcal{R}_{\mathbold{r}}:=  \left(\int_{A_{[n]-J}}
R^{J}_{\mathfrak{r}_{J}}\right)^{\prec}=
R^{(n)}_{\mathbold{r}(n)}
\cdots R^{(2)}_{\mathbold{r}(2)}
R^{(1)}_{\mathbold{r}(1)}\,.
\end{equation}
Again we call the data $\mathbold{r}$ in \eqref{eqndfnforest} associated to $(J,\mathfrak{r}_{J})$  in Lemma \ref{lemhollimitintermsofforest} a forest, and any element in $[n]$ a vertex.
By defining an equivalence relation $\sim $ on $[n]$ by declaring
\[
j\sim \mathbold{r}(j)\,,\quad j\in J\,,
\]
the set $[n]$ is decomposed into several disjoint subsets called trees\footnote{
The combinatorial data of $\mathbold{r}$
as a labeled rooted  forest $F$ is discussed in \cite{Li:2020regularized},  with the root labeled by $0$. The number $C_{J,\mathfrak{r}_{J}}$ is
the quantity $p(F)$ in \cite[Lemma 3.36]{Li:2020regularized}, called tree factorial in \cite{Kreimer:2000}.
}
\begin{equation}\label{eqndecompositionintotrees}
[n]=\bigsqcup_{k\in \mathbold{r}^{-1}(0)} \Gamma_{k}\,.
\end{equation}
It is clear  that the index set $\mathbold{r}^{-1}(0)$
satisfies $|\mathbold{r}^{-1}(0)|=1$
if and only if $|\mathfrak{r}_{J}-J|=1$, in which case the forest itself is a tree.\\

The following result Lemma \ref{lemresidueoperatorsassociatedtoforests} relates the residue operators $\mathcal{R}_{\mathbold{r}} $ associated  to trees and forests
to elements in the free algebra generated by the operator variables $R^{(i)}_{0},i\in [n]$.

\begin{lem}\label{lemresidueoperatorsassociatedtoforests}
Let the notations be as above.
\begin{enumerate}[(i).]
\item
(Residue operator associated to a tree)\label{lemnestedcommutator}
Assume $2\leq m\leq n$.
Let $\mathfrak{r}=(\mathfrak{r}_{1},\cdots, \mathfrak{r}_{m-1})$
be a sequence satisfying
\[
\mathfrak{r}_{a}\in [m]\,,~ \mathfrak{r}_{m-1}=m\,,\quad\quad a<\mathfrak{r}_{a}\,,~\text{for}~a=1,2,\cdots, m-1\,.
\]
Then for any $c\in [n]\cup\{0\}-[m]$, the operator on quasi-elliptic functions
\[
R^{(m)}_{c}\circ
R^{[m-1]}_{\mathfrak{r}}
:=R^{(m)}_{c}\circ
R^{(m-1)}_{m}\circ \cdots R^{(2)}_{\mathfrak{r}_{2}}\circ R^{(1)}_{\mathfrak{r}_{1}}
\]
is sum of nested commutators of $m$ residue operators of the form $R^{(a)}_{c}, a\in [m] $.

\item (Residue operator associated to a forest)\label{lemLieelement}
The operator
$\mathcal{R}_{\mathbold{r}}$ in \eqref{eqndfnresidueoperator} is
a product $L_{v}\cdots L_{1}$ of $v=|\mathbold{r}^{-1}(0)|$ factors $L_{k},k=1,2,\cdots, v$,
with the factors being the operators associated to the trees in the decomposition \eqref{eqndecompositionintotrees}.
In particular, each $L_{k}$ is
 either of the form $R^{(i)}_{0},i\in [n]$ or a sum of nested commutators
of $m_k$ such operators, where $m_k$ is the number of vertices  of the tree corresponding to $L_{k}$.\footnote{Namely, it is a  homogeneous Lie element in the free Lie algebra generated  by the set of operator variables $R^{(a)}_0, a\in [n]$.}
\end{enumerate}
\end{lem}
 The constructive proof of  Lemma \ref{lemresidueoperatorsassociatedtoforests} is essentially contained in
 \cite{Li:2020regularized}.
 To prove Lemma \ref{lemresidueoperatorsassociatedtoforests} \eqref{lemnestedcommutator}, we first show the statement for operators of the form $R^{(b)}_{c}\prod_{i}R^{(i)}_{b}$ using Lemma \ref{lemcommutatorresiduec=0},
  then use
 induction on the height (i.e., the smallest number $h$ such that $\mathbold{r}^{h}(j)=0$) of the vertices $j\in [n]$.
 After that, one applies Lemma \ref{lemcommutatorresiduec=0} to move the residue operators $R^{(a)}_{b},b\neq 0$ across
 the $R^{(e)}_{0}$ operators and reduce  the discussions on forests  in Lemma   \ref{lemresidueoperatorsassociatedtoforests} \eqref{lemLieelement} to
 those on trees in Lemma \ref{lemresidueoperatorsassociatedtoforests} \eqref{lemnestedcommutator}.
 The details are
 elementary but tedious, and  are therefore omitted.

\begin{rem}

We can in fact show that
the operator $R^{(m)}_{c}\circ
R^{[m-1]}_{\mathfrak{r}}$ in  Lemma \ref{lemresidueoperatorsassociatedtoforests}  \eqref{lemnestedcommutator}
and each $L_{k}$ in
Lemma  \ref{lemresidueoperatorsassociatedtoforests} \eqref{lemLieelement}, provided that the total number $m_{k}$ of vertices in the corresponding tree is at least $2$,
can be expressed as a nested commutator instead of a sum of nested commutators. A graphic rule, which is essentially given by the blackboard orientation, can be found in \cite[Section 3.4]{Li:2020regularized}.

Furthermore, using Lemma \ref{lemcommutatorresiduec=0} one can show that the compositions of residue operators $\mathcal{R}_{\mathbold{r}}$
can be conveniently described using the gluing of labeled rooted  forests in the indicated orderings.
Lemma \ref{lemcommutatorresiduec=0}
can then be interpreted as saying that the
algebra generated by the residue operators $R^{(a)}_{b},a\in [n],b\in [n]\cup\{0\}$ is acted on by the underlying operad of labeled rooted  forests.

\end{rem}

\begin{proof}[Another proof of Proposition \ref{propweightofoperators} (ii)]

Consider the average of \eqref{eqnholomorphiclimitregularizedintegral}
over $\sigma\in \mathfrak{S}_{n}$.
Observe that $C_{J,\mathfrak{r}_{J}}$ is independent of  the ordering of
elements in the underlying set of $J\cup\mathfrak{r}_{J}$, as can be seen from its expression.
Thanks to this property,
to prove the desired statement Proposition \ref{propweightofoperators} (ii) it is enough
to show that\footnote{This is essentially \cite[Lemma 3.40]{Li:2020regularized}. The following argument  
 also gives a shorter proof of
\cite[Lemma 3.42]{Li:2020regularized}.}
\begin{equation}\label{eqnaverageleadstocancellation}
\sum_{\sigma\in \mathfrak{S}_{n}}\mathcal{R}_{\sigma(\mathbold{r})}=0\,,\quad ~\text{if}\quad\mathbold{r}^{-1}([n])\neq \emptyset\,.
\end{equation}
Here as in \eqref{eqnsigmaregularizedintegralassumoverr} and \eqref{eqnsumoverrelabelingsolabeledrootedforests}, we have $\mathcal{R}_{\sigma(\mathbold{r})}=
R^{(\sigma([n-1]))}_{\sigma(\mathbold{r})(\sigma([n-1]))},\sigma(\mathbold{r})=\sigma\circ \mathbold{r}\circ\sigma^{-1}$ with the convention $\sigma(0)=0$.

For any such $\mathbold{r}$,
applying Lemma \ref{lemresidueoperatorsassociatedtoforests} \eqref{lemLieelement}
we obtain a decomposition $\mathcal{R}_{\mathbold{r}}=L_{v}\cdots L_1$.
The constraint on $\mathbold{r}$ in Lemma  \ref{lemhollimitintermsofforest} implies that the number $v=|\mathbold{r}^{-1}(0)|$ of trees in the decomposition \eqref{eqndecompositionintotrees}
satisfies $v\leq n-1$.
This tells that
there is at least a factor above, say $L_{*}$, such that
the number $m$ of vertices of the corresponding tree $T$ is
at least $2$.
Therefore, $L_{*}$
is a sum of $m$-nested commutators.
Denote the set of vertices of the tree $T$ by $V$.
The summation over $\mathfrak{S}_{n}$ can be organized into a multi-summation similar to \eqref{eqnsumovergammarhosigma}: over the choices for the
underlying set of $V$,
over the orderings of elements
in the underlying set of $V$, and over the orderings of the rest of the elements in
$\{1,2,\cdots, n\}$.
Now to prove \eqref{eqnaverageleadstocancellation}
it suffices to show that
\begin{equation}\label{eqnaverageleadstocancellationnested}
\sum_{\sigma\in \mathfrak{S}(V)}\sigma(L_{*})=0\,,
\end{equation}
where $\sigma(L_{*})$ is the operator obtained by permuting the indices of the $m$-nested residue operators in $L_{*}$
using $\sigma$ from the permutation group $\mathfrak{S}(V)$.

Let $K$ be any
$m$-nested commutator summand of $L_{*}$.
Consider the upper-indices, say $a_1,a_2$, of an inner-most commutator of $K$.
The summation $\sum_{\sigma\in \mathfrak{S}(V)}\sigma(K)$
above
 is a multi-summation, with one of them being the sum over the
permutation group of $\{a_1,a_2\}$.
The result  is therefore zero by the anti-symmetry\footnote{
In \cite[Lemma 3.40, Lemma 3.42]{Li:2020regularized}, the vanishing is proved by using quite involved combinatorics on labeled rooted  forests.
Here the proof is simplified by using only the nested commutator structure and the resulting anti-symmetry.
} of the nested commutator $K$ in $a_1,a_2$.
By linearity, this then implies  \eqref{eqnaverageleadstocancellationnested}.
The proof is now complete.

\end{proof}

\end{appendices}

\bibliographystyle{amsalpha}

\providecommand{\bysame}{\leavevmode\hbox to3em{\hrulefill}\thinspace}
\providecommand{\MR}{\relax\ifhmode\unskip\space\fi MR }
\providecommand{\MRhref}[2]{%
  \href{http://www.ams.org/mathscinet-getitem?mr=#1}{#2}
}
\providecommand{\href}[2]{#2}

\bigskip{}

\noindent{\small Yau Mathematical Sciences Center, Tsinghua University, Beijing 100084, P. R. China}

\noindent{\small Email: \tt  sili@mail.tsinghua.edu.cn}


\noindent{\small Email: \tt jzhou2018@mail.tsinghua.edu.cn}

\end{document}